\documentclass[10pt]{amsart}

\usepackage{amsmath}
\usepackage{mathrsfs}
\usepackage{amssymb}
\usepackage{amsfonts}
\usepackage{amsthm}
\usepackage{enumerate}
\usepackage[top=1.5in, bottom=1.5in, left=1.25in, right=1.25in]{geometry}
\usepackage{mathtools}
\usepackage{esint}
\usepackage[hidelinks]{hyperref}
\usepackage{xcolor}
\usepackage{comment}

\usepackage[style=trad-plain, hyperref=true]{biblatex}
\providecommand{\N}{\mathbb{N}}
\providecommand{\Z}{\mathbb{Z}}

\providecommand{\R}{\mathbb{R}}
\providecommand{\C}{\mathbb{C}}

\providecommand{\E}{\mathbb{E}}

\providecommand{\T}{\mathbb{T}}
\providecommand{\V}{\mathcal{V}}
\providecommand{\F}{\mathcal{F}}
\providecommand{\D}{\mathbb{D}}
\providecommand{\M}{\mathcal{M}}
\providecommand{\leqsim}{\lesssim}
\providecommand{\geqsim}{\gtrsim}

\renewcommand{\:}{\colon}

\renewcommand{\S}{\mathcal{S}}

\DeclareMathOperator{\supp}{supp}

\numberwithin{equation}{section}

\newtheorem{theorem}{Theorem}[section]
\newtheorem{proposition}[theorem]{Proposition}
\newtheorem{corollary}[theorem]{Corollary}
\newtheorem{lemma}[theorem]{Lemma}

\theoremstyle{definition}

\theoremstyle{remark}
\newtheorem*{remark}{Remark}

\makeatletter
\DeclareRobustCommand\widecheck[1]{{\mathpalette\@widecheck{#1}}}
\def\@widecheck#1#2{%
    \setbox\z@\hbox{\m@th$#1#2$}%
    \setbox\tw@\hbox{\m@th$#1%
       \widehat{%
          \vrule\@width\z@\@height\ht\z@
          \vrule\@height\z@\@width\wd\z@}$}%
    \dp\tw@-\ht\z@
    \@tempdima\ht\z@ \advance\@tempdima2\ht\tw@ \divide\@tempdima\thr@@
    \setbox\tw@\hbox{%
       \raise\@tempdima\hbox{\scalebox{1}[-1]{\lower\@tempdima\box
\tw@}}}%
    {\ooalign{\box\tw@ \cr \box\z@}}}
\makeatother

\title{Variational Estimates for Bilinear Ergodic Averages Along Sublinear Sequences}
\author{Maximilian O'Keeffe}
\date{}

\begin{document}
\maketitle

\begin{abstract}
    We prove long variational estimates for the bilinear ergodic averages \[ A_{N;X}(f,g)(x) = \frac{1}{N} \sum_{n=1}^N f(T^{\lfloor \sqrt{n} \rfloor}x) g(T^nx) \] on an arbitrary measure preserving system $(X,\mu,T)$ for the full expected range, i.e. whenever $f \in L^{p_1}(X)$ and $g \in L^{p_2}(X)$ with $1<p_1,p_2<\infty$. In particular, if $\frac{1}{p}=\frac{1}{p_1}+\frac{1}{p_2}$ we show that the long $r$-variation of $A_{N;X}$ maps $L^{p_1}(X) \times L^{p_2}(X)$ into $L^p(X)$ for any $p>\frac{1}{2}$, which is sharp up to the endpoint. If $p \geq 1$ we obtain long variational estimates for the full expected range $r>2$ and if $p<1$ we obtain a range of $r>2+\varepsilon_{p_1,p_2}$ where $\varepsilon_{p_1,p_2}>0$ depends only on $p_1$ and $p_2$. As a consequence, we obtain bilinear maximal estimates \[ \left\| \sup_{N \in \N} |A_{N;X}(f,g)| \right\|_{L^p(X)} \leq C_{p_1,p_2} \|f\|_{L^{p_1}(X)} \|g\|_{L^{p_2}(X)} \] for any $1<p_1,p_2 \leq \infty$.
\end{abstract}

\tableofcontents

\section{Introduction}

\subsection{Non-Conventional Ergodic Averages}

Let $(X,\mu,T)$ be a \emph{measure preserving system}, so that $(X,\mu)$ forms a measure space (note we make no assumption here that $X$ has finite measure) and $T\:X \to X$ is an invertible bi-measurable map which is \emph{measure preserving} in the sense that $\mu(T(A))=\mu(A)$ for all measurable sets $A$. Let $f,g\:X \to \C$ be measurable. Then for $N \in\N$ one can form the so-called \emph{non-conventional} ergodic average (using terminology  introduced in \cite{Furstenberg88}) \begin{equation} \label{eqn:averagedef} A_{N;X}(f,g)(x) = \frac{1}{N} \sum_{n=1}^N f(T^{\lfloor \sqrt{n} \rfloor}x) g(T^n x). \end{equation}

The study of the convergence of ergodic averages as $N \to \infty$ was initiated by Birkhoff \cite{Birkhoff89} and von Neumann \cite{neumann1932proof}. Birkhoff's ergodic theorem states that for all $f \in L^p(X)$ with $p \geq 1$ the ergodic averages \[ \frac{1}{N} \sum_{n=1}^N f(T^nx) \] converge pointwise, for $\mu$-almost every $x \in X$, as $N \to \infty$ and von Neumann proved that they converge in $L^p(X)$ norm. These results have applications in ergodic theory, number theory, and additive combinatorics.

Birkhoff's and von Neumann's theorems were first generalised by Furstenberg \cite{furstenberg1977ergodic} in his ergodic theoretic proof of Szemerédi's theorem \cite{szemeredi1975sets}, which states that any subset $E$ of the natural numbers of \emph{positive upper density}, i.e. \[ \limsup_{N \to \infty} \frac{|E \cap \{1, \dotsc, N\}|}{N}>0 \] contains arithmetic progressions of arbitrary length. Thanks to a correspondence principle in that paper, one can deduce the existence of patterns of the form \[ \{x, x+P_1(n),\dotsc,P_m(n)\} \] in subsets of $\Z$ by proving that ergodic averages of the form \begin{equation} \label{eqn:multipleergodicaverages} \frac{1}{N} \sum_{n=1}^N \prod_{i=1}^m f_i(T^{P_i(n)} x) \end{equation} in $L^2(X)$ norm for any $f_1,\dotsc, f_m \in L^\infty(X)$. One main area of focus was to take $P_1, \dotsc, P_m$ to be polynomials with integer coefficients. Many such norm convergence results were proved in this direction \cite{furstenberg1977ergodic, bergelsonpet, furstenberg1996mean, host2005nonconventional} and this particular area of investigation concluded in 2005 when Leibman proved that (\ref{eqn:multipleergodicaverages}) converge in $L^2(X)$ norm for any polynomials $P_1,\dotsc,P_m$ \cite{leibman2005convergence}.

On the other hand, results about pointwise convergence have been much more limited historically. Bourgain had the first breakthrough in 1988 and 1989 in a series of papers \cite{bourgain1988maximal, bourgain1988pointwise, Bourgain89} in which he showed that \[ \frac{1}{N} \sum_{n=1}^N f(T^{P(n)}x) \] converges whenever $P$ has degree at least 2 and $f \in L^p(X)$ with $1<p<\infty$. In 1990 he then proved \cite{Bourgain90} that for any bounded functions $f$ and $g$ the averages \[ \frac{1}{N} \sum_{n=1}^N f(T^nx) g(T^{kn}x) \] where $k$ is a non-zero integer. For pointwise convergence of averages of the form (\ref{eqn:multipleergodicaverages}), nothing else was known until until 2022 when Krause, Mirek, and Tao \cite{KMT} established the pointwise convergence of the bilinear ergodic averages \[ \frac{1}{N} \sum_{n=1}^N f(T^nx) g(T^{P(n)}x) \] where $P \in \Z[n]$ is a polynomial of degree at least 2. This was achieved by proving a so-called long variational inequality.

The aim of this paper is to establish a long variational inequality for $A_{N;X}(f,g)$ whenever $f \in L^{p_1}(X)$ and $g \in L^{p_2}(X)$ with $1<p_1,p_2<\infty$. As a consequence we will have proved a maximal inequality for these averages. We emphasise that the range $1<p_1,p_2<\infty$ is the full expected range.

We remark here that in the case where $X$ has finite measure and $f,g \in L^\infty(X)$, the pointwise convergence of $A_{N;X}(f,g)$ was proved by Donoso, Koutsogiannis, and Sun in \cite{donoso2020pointwise}, and an unpublished argument of Wierdl establishes the pointwise convergence of $A_{N;X}(f,g)$ on an arbitrary measure space by using maximal inequalities of the pertaining linear operators and density arguments. Moreover, in this case the norm convergence of $A_{N;X}(f,g)$ was established by Frantzikinakis \cite{frantzikinakis2010multiple}. Our long variational inequality therefore will provide a new proof of pointwise and norm convergence although this is heavy machinery for such a task. Instead the perspective we take is that variational inequalities provide information on the rates of convergence for such averages, and this will be new. We note however that if one replaced the sequences $\sqrt{n}$ and $n$ by sparser sequences, such as $n^c$ with $c>1$ non-integer, then Wierdl's argument does not apply but proving a long variational inequality does.

\subsection{Statement of Results}

One method of proving that $A_{N;X}(f,g)$ converges pointwise is to prove a maximal inequality of the form \begin{equation} \label{eqn:maximalineq} \left\| \sup_{N \in \N} |A_{N;X}(f,g)| \right \|_{L^p(X)} \leqsim_{p_1,p_2,p} \|f\|_{L^{p_1}(X)} \|g\|_{L^{p_2}(X)} \end{equation} for some $p>0$ and then to establish pointwise convergence for a dense subclass of functions. This can be used, for example, to prove that Birkhoff's ergodic averages converge pointwise. The issue with this method when dealing with non-conventional averages is it becomes less clear what dense subclass of functions one should take.

An alternative method is to prove a so-called \emph{$r$-variational inequality}. For each $x \in X$ one has a sequence $(A_{N;X}(f,g)(x))_{N \in \N}$ of complex numbers. We can therefore define the $r$-variation norm of this sequence as \begin{equation} V^r(A_{N;X}(f,g)(x))_{N \in \N} = \sup_{N \in \N} |A_{N;X}(f,g)(x)| + \V^r(A_{N;X}(f,g)(x))_{N \in \N}  \end{equation} where \begin{equation} \V^r(A_{N;X}(f,g)(x))_{N \in \N} = \sup_{J \in \N} \sup_{\substack{N_0 \leq \dotsb \leq N_J \\ N_j \in \N}} \left( \sum_{j=0}^{J-1} |A_{N_{j+1};X}(f,g)(x) - A_{N_j;X}(f,g)(x)|^r \right)^\frac{1}{r} \end{equation} if $0<r<\infty$ and where \begin{equation} \V^\infty(A_{N;X}(f,g)(x))_{N \in \N} = \sup_{M<N} |A_{M;X}(f,g)(x)-A_{N;X}(f,g)(x)|. \end{equation} Note in these definitions one could replace $\N$ with some other set $\D$ so that we take $N,N_0, \dotsc, N_J \in \D$. This norm is relevant in the study of pointwise convergence because if \begin{equation} V^r(A_{N;X}(f,g)(x))_{N \in \N}<\infty \label{eqn:finitevariation} \end{equation} then the sequence $A_{N;X}(f,g)(x)$ converges as $N \to \infty$. The map \[ x \mapsto V^r(A_{N;X}(f,g)(x))_{N \in \N} \] is measurable so its $L^p(X)$ norm is well-defined. If one can show that this $L^p(X)$ norm is finite for some $p$ and $r$, then we will have (\ref{eqn:finitevariation}) for $\mu$-almost every $x \in X$, and hence we will have pointwise convergence almost everywhere. Thus proving an $r$-variational inequality would involve proving an inequality of the form \[ \left\| V^r(A_{N;X}(f,g))_{N \in \N} \right \|_{L^p(X)} \leqsim_{p_1,p_2,p,r} \|f\|_{L^{p_1}(X)} \|g\|_{L^{p_2}(X)} \] for some $p>0$ and $r\geq 1$.

The benefit of an $r$-variational inequality is that it implies the maximal inequality (\ref{eqn:maximalineq}) and it implies both norm and pointwise convergence for all $f \in L^{p_1}(X)$ and $g \in L^{p_2}(X)$. This remains true if instead of taking $N \in \N$, one takes $N$ in an arbitrary $\lambda$-lacunary subset $\D \subseteq \N$ where $\lambda>1$ (thus for all $N_1,N_2 \in \D$ with $N_2>N_1$ we have $\frac{N_2}{N_1}>\lambda$). In this case, an inequality of the form \[  \left\| V^r(A_{N;X}(f,g))_{N \in \D} \right \|_{L^p(X)} \leqsim_{p_1,p_2,p,r,\lambda} \|f\|_{L^{p_1}(X)} \|g\|_{L^{p_2}(X)} \] whenever $\D$ is $\lambda$-lacunary is known as a \emph{long variational inequality}. This method and terminology was used in \cite{KMT} where the authors established the pointwise convergence of the ergodic averages \[ \frac{1}{N} \sum_{n=1}^N f(T^nx) g(T^{P(n)}x) \] where $P \in \Z[n]$ is a polynomial of degree at least 2.

Thus our main result can be stated as follows.

\begin{theorem} \label{thm:mainthm}
Let $(X,\mu,T)$ be a measure preserving system. Let $f \in L^{p_1}(X)$ and $g \in L^{p_2}(X)$ with $1<p_1,p_2<\infty$ and define $\frac{1}{p} = \frac{1}{p_1} + \frac{1}{p_2}$. Let \[ c_{p_1,p_2} = \begin{cases} \max(p_1',p_2') & p<1, \\ 2 & p \geq 1. \end{cases} \] For any $r>c_{p_1,p_2}$ and $\lambda>1$ we have the long variational inequality \begin{equation} \|V^r(A_{N;X}(f,g))_{N \in \D} \|_{L^p(X)} \leqsim_{p_1,p_2,r,\lambda} \|f\|_{L^{p_1}(X)} \|g\|_{L^{p_2}(X)} \label{eqn:variationineq} \end{equation} whenever $\D \subseteq \N$ is $\lambda$-lacunary.
\end{theorem}

As mentioned above, Theorem \ref{thm:mainthm} implies a maximal inequality and both pointwise and norm convergence.

\begin{corollary} \label{cor:normmaximal}
Let $(X,\mu,T)$ be a measure preserving system. Let $f \in L^{p_1}(X)$ and $g \in L^{p_2}(X)$ with $1<p_1,p_2<\infty$ and define $\frac{1}{p} = \frac{1}{p_1}+\frac{1}{p_2}$. Then one has the maximal inequality \begin{equation} \label{eqn:maximal} \left\| \sup_{N \in \N} |A_{N;X}(f,g)| \right\|_{L^p(X)} \leqsim_{p_1,p_2} \|f\|_{L^{p_1}(X)} \|g\|_{L^{p_2}(X)}, \end{equation} and the averages $A_{N;X}(f,g)$ converge both pointwise almost everywhere and in $L^p(X)$ norm as $N \to \infty$.
\end{corollary}

\subsubsection{Reductions} \label{sec:reductions}

As in \cite{KMT} there are a couple of reductions we can make which we list here.

\begin{enumerate}[(i)]
    \item One of the assumptions of Theorem \ref{thm:mainthm} is that $\frac{1}{p_1}+\frac{1}{p_2}=\frac{1}{p}$. Thus the Calderón transference principle \cite{Transference} applies, so to prove that (\ref{eqn:variationineq}) holds for an arbitrary measure preserving system it suffices to prove it in the case $(X,\mu,T) = (\Z,\mu_\Z,T_\Z)$ where $\mu_\Z$ is the counting measure on $\Z$ and $T_\Z\:\Z \to \Z$ is the \emph{shift map} defined by $T_\Z(x)=x-1$. In this case we will denote $A_{N;\Z}$ simply by $A_N$. Note in this case we have \[ A_N(f,g)(x) = \frac{1}{N} \sum_{n=1}^N f(x-\lfloor \sqrt{n} \rfloor) g(x-n). \]
    
    \item In order to prove (\ref{eqn:variationineq}) it suffices to replace $A_{N;X}$ by the so-called \emph{upper-half operator} (using terminology from \cite{KMT}) defined by \[ \Tilde{A}_{N;X}(f,g)(x) = \frac{1}{N} \sum_{n=1}^N 1_{n>\frac{N}{2}} f(T^{\lfloor \sqrt{n} \rfloor} x) g(T^nx). \]
\end{enumerate}

\begin{remark}
In \cite{KMT}, another reduction is made: if one can prove that (\ref{eqn:variationineq}) holds for all $\lambda>0$ and all \emph{finite} $\lambda$-lacunary subsets $\D \subseteq \N$ then Theorem \ref{thm:mainthm} follows. Since our proof does not rely on the finiteness of $\D$ at any point we will not need this reduction.
\end{remark}

Combining these two reductions, we see that Theorem \ref{thm:mainthm} is implied by the following theorem.

\begin{theorem} \label{thm:reducedmainthm}
Let $f \in \ell^{p_1}(\Z)$ and $g \in \ell^{p_2}(\Z)$ with $1<p_1,p_2<\infty$ and define $\frac{1}{p} = \frac{1}{p_1} + \frac{1}{p_2}$. Let \[ c_{p_1,p_2} = \begin{cases} \max(p_1',p_2') & p<1, \\ 2 & p \geq 1. \end{cases} \] For any $r>c_{p_1,p_2}$ and $\lambda>1$ we have the long variational inequality \begin{equation} \|V^r(\Tilde{A}_N(f,g))_{N \in \D} \|_{\ell^p(\Z)} \leqsim_{p_1,p_2,r,\lambda} \|f\|_{\ell^{p_1}(\Z)} \|g\|_{\ell^{p_2}(\Z)} \label{eqn:reducedvariationineq} \end{equation} whenever $\D \subseteq \N$ is $\lambda$-lacunary.
\end{theorem}

The paper is dedicated to proving Theorem \ref{thm:reducedmainthm}. From now on we will fix $p_1,p_2,p,r,\lambda$ and assume they satisfy the hypotheses of Theorem \ref{thm:reducedmainthm}. All implied constants will be allowed to depend on these parameters. We will also fix the finite $\lambda$-lacunary subset $\D$ but we do not allow implied constants to depend on $\D$.

\subsubsection{Sharpness}

Before moving on to the proof of Theorem \ref{thm:reducedmainthm} we briefly discuss the sharpness of Theorem \ref{thm:mainthm}, Corollary \ref{cor:normmaximal}, and Theorem \ref{thm:reducedmainthm}.

We will show that Theorem \ref{thm:mainthm} is sharp up to the endpoint in the following sense.

\begin{proposition} \label{prop:sharpness}
Fix $p_1,p_2>0$ and define $\frac{1}{p}=\frac{1}{p_1}+\frac{1}{p_2}$. Let $(X,\mu)$ be a non-atomic probability space and let $T\:X \to X$ be ergodic. Suppose one has (\ref{eqn:maximal}) for all $f \in L^{p_1}(X)$ and $g \in L^{p_2}(X)$. Then necessarily $p \geq \frac{1}{2}$.
\end{proposition}

\begin{proof}
Suppose \[ \| \sup_{N \in \N} |A_{N;X}(f,g)| \|_{L^p(X)} \leqsim_{p_1,p_2} \|f\|_{L^{p_1}(X)} \|g\|_{L^{p_2}(X)} \] holds for all $f \in L^{p_1}(X)$ and $g \in L^{p_2}(X)$. Take a sequence $(A_k)$ of measurable sets of positive measure such that $\mu(A_k) \to 0$. Fix $k$ and let $f=g=1_{A_k}$. Then $\|1_{A_k}\|_{L^q(X)} = \mu(A_k)^\frac{1}{q}$ for any $q>0$ so for all $N \in \N$ we have \[ \|\E_{n \in [N]} T^{\lfloor \sqrt{n} \rfloor} 1_{A_k} \cdot T^n1_{A_k} \|_{L^p(X)} \leqsim_{p_1,p_2} \mu(A_k)^{\frac{1}{p_1}+\frac{1}{p_2}} = \mu(A_k)^\frac{1}{p}. \] Since $X$ is a probability space and $T$ is ergodic, we have \[ \E_{n \in [N]} T^{\lfloor \sqrt{n} \rfloor} 1_{A_k} \cdot T^n1_{A_k} \to \mu(A_k)^2 \] almost everywhere by Corollary \ref{cor:equidistribution}. By dominated convergence, it follows that $\mu(A_k)^2 \leqsim \mu(A_k)^\frac{1}{p}$. Taking $k \to \infty$ yields the claim.
\end{proof}

Proposition \ref{prop:sharpness} immediately shows that the maximal inequality (\ref{eqn:maximal}) in Corollary \ref{cor:normmaximal} fails in general if $p<\frac{1}{2}$. Since, for a fixed choice of $1<p_1,p_2<\infty$, (\ref{eqn:variationineq}) implies (\ref{eqn:maximal}), we see that if $p<\frac{1}{2}$ then Theorem \ref{thm:mainthm} also fails. Moreover, Theorem \ref{thm:reducedmainthm} implies \ref{thm:mainthm} for all measure preserving systems by the Calder\'on transference principle so if $p<\frac{1}{2}$ then Theorem \ref{thm:reducedmainthm} fails too.

\subsection{Overview of Proof}

Our proof will use the method employed in \cite{KMT}.  Working on the integers will allow tools from harmonic analysis to be employed. However, in \cite{KMT} the ergodic average is sampled along the sequences $n$ and $P(n)$, which are both polynomial while in this paper the sequences $\lfloor \sqrt{n} \rfloor$ and $n$ are instead rational functions (or one could even view them as coming from a Hardy field). This leads to some key differences between the proofs.

These differences stem from the relevant exponential sums one needs to consider when looking at averaging operators. Note that if $P(n)$ and $Q(n)$ are sequences then \[ \frac{1}{N} \sum_{n=1}^N f(x-P(n))g(x-Q(n)) = \sum_{y_1,y_2} K(y_1,y_2) f(x-y_1) g(x-y_2) \] where \[ K(y_1,y_2) = \int_{\T^2} \frac{1}{N} \sum_{n=1}^N e^{2\pi i(\zeta P(n) + \xi Q(n))} e(\zeta y_1+\xi y_2) \,d\zeta \,d\xi. \] The study of averaging operators therefore necessarily involves the study of the corresponding exponential sums. In particular, an analysis of when these exponential sums are ``large" is important. For the averages considered in \cite{KMT}, note that if $P$ is a polynomial of degree $d \geq 2$ with leading coefficient $a_d$ and \[ \delta \leq \left| \frac{1}{N} \sum_{n=1}^N e^{2\pi i(\zeta n + \xi P(n)} \right| \] then one can deduce that there exists $q \leqsim \delta^{-O_d(1)}$ such that $\| q a_d \xi\|_\T \leqsim \delta^{-O_d(1)} N^{-d}$. In other words, if the exponential sum is large then $\xi$ lies in an arc centred at some rational with denominator at most $\frac{1}{q a_d}$ of width $O(\delta^{-O_d(1)} N^{-d})$ (in other words, $\xi$ lies in the set of \emph{major arcs}). In this paper, we are able to reduce even further, and show that if \[ \delta \leq \left| \frac{1}{N} \sum_{n=1}^N e^{2\pi i(\zeta \lfloor \sqrt{n} \rfloor + \xi n)} \right| \] then $\|\xi\|_\T \leqsim \delta^{-O(1)} N^{-1}$. Thus if the relevant exponential sum is large, then $\xi$ lies in an arc of width $O(\delta^{-O(1)} N^{-1})$ centred at the origin, and not just in an arc centred at \emph{some} rational. In this case we would say that $\xi$ lies in the \emph{principal} major arc, or the major arc centred at 0.

As a consequence, in \cite{KMT} the authors reduce their analysis to the case where $f$ and $g$ have Fourier transforms which are supported on the set of major arcs, and multifrequency analysis is needed. In this paper we will instead be able to reduce even further to the case where $f$ and $g$ have Fourier transforms which are supported on the principal major arc. The benefit is then that single frequency analysis can be utilised instead of multifrequency analysis. %See Example \ref{eg:majorarcbehaviour} for an illustration of this phenomenon.
The payoff of the work we do to reduce to the principal major arc case is that we will be able to establish pointwise convergence for the full expected range, i.e. we will establish pointwise convergence of $A_{N;X}(f,g)$ for all $f \in L^{p_1}(X)$ and $g \in L^{p_2}(X)$ with $1<p_1,p_2<\infty$.

Another consequence of the difference in the exponential sum bounds is in the minor arc estimate. The main tool used in \cite{KMT} is the following inverse theorem of Peluse \cite{peluse} from additive combinatorics.

\begin{theorem}[Peluse] \label{thm:peluseinverse}
Let $N \in \N$ and $0<\delta \leq 1$. Let $P\:\Z \to \Z$ be a polynomial of degree $d \geq 2$. Suppose $f_0,f_1,f_2\:\Z \to \C$ are $1$-bounded functions supported on $[-N_0,N_0]$ with $N_0 \sim N^d$ and assume \[ \delta N^d \leq \left| \sum_{x \in \Z} f_0(x) \cdot \frac{1}{N} \sum_{n=1}^N f_1(x-n) f_2(x-P(n)) \right|. \] Then either $N \leqsim \delta^{-O(1)}$ or there exists $q \leqsim \delta^{-O(1)}$ and $N' \in \N$ with $\delta^{O(1)}N \leqsim N' \leq N$ such that \[ \sum_{x \in \Z} \left| \frac{1}{N'} \sum_{n=1}^{N'} f_1(x+qn) \right| \geqsim \delta^{O(1)} N^d. \]
\end{theorem}

This theorem roughly states that if the bilinear average \[ \frac{1}{N} \sum_{n=1}^N f_1(x-n) f_2(x-P(n)) \] correlates highly with some $1$-bounded function then the Fourier transform of $f_1$ has a large contribution on some major arc. By the previous discussion, an analogous result in this paper should show that the Fourier transform of $f_1$ has a large contribution on the principal major arc, or in other words we should have the following.

\begin{theorem} \label{thm:inverse} Let $N \in \N$ and $0<\delta \leq 1$. Suppose $f_0,f_1,f_2\:\Z \to \C$ are $1$-bounded functions supported on $[-N_0,N_0]$ with $N_0 \sim N$ and suppose \[ \delta N \leq \left| \sum_{x \in \Z} f_0(x) \cdot \frac{1}{N} \sum_{n=1}^N f_1(x-\lfloor \sqrt{n} \rfloor) f_2(x-n) \right|. \] Then either $N \leqsim \delta^{-O(1)}$ or there exists $\delta^{O(1)} N \leqsim N' \leq N$ such that \[ \sum_{x \in \Z} \left| \frac{1}{N'} \sum_{n=1}^{N'} f_1(x-n) \right| \geqsim \delta^{O(1)} N \]
\end{theorem}

One might expect that a proof of Theorem \ref{thm:inverse} could be done without any appeal to Peluse's results, which deal entirely with polynomial sequences. Indeed, fractional monomials are smooth, slowly varying, and only the principal major arc plays a role when considering exponential sums sampled along them. However, it seems necessary to perform a ``change of variables" to approximate the average \[ \left| \sum_{x \in \Z} f_0(x) \cdot \frac{1}{N} \sum_{n=1}^N f_1(x-\lfloor \sqrt{n} \rfloor) f_2(x-n) \right| \] by a polynomial average \[ \left| \sum_{x \in \Z} f_0(x) \cdot \frac{1}{N} \sum_{n=1}^N f_1(x-m) f_2(x-m^2-n) \right| \] for some $n$. We could at this stage also cite Theorem \ref{thm:peluseinverse}. However, instead we do some more work and cite an earlier result in \cite{peluse} which is itself used to prove Theorem \ref{thm:peluseinverse}. By modifying Peluse's proof (using the principal major arc behaviour of our exponential sums) we are able to reduce further and obtain principal major arc behaviour in the Fourier transform of $f_1$.

\section{Rates of Convergence} \label{sec:examples}

In this section we quickly demonstrate how our main results provide quantitative information on the rate of convergence of ergodic averages. For any unfamiliar notation see Section \ref{sec:preliminaries}.

We introduce the \emph{jump counting function}. Given $\delta>0$ and an indexing set $\D$ we let $\mathcal{N}_\delta(a_N)_{N \in \D}$ denote the number of $\delta$ jumps of the sequence $(a_N)$ with times in $\D$. More precisely, $\mathcal{N}_\delta(a_N)_{N \in \D}$ is the supremum over all $J \in \N$ for which there exist times $N_0 \leq \dotsb N_J$, with $N_j \in \D$, such that $|a_{N_{j+1}}-a_{N_j}| \geq \delta$ for each $j \in \{0, \dotsc, J-1 \}$. Bounds on $\mathcal{N}_\delta(a_N)_{N \in \D}$ provide quantitative information on the rate of convergence of the sequence $(a_N)$.

Given any such set of indices $\D$, $r>0$, and sequence $(a_N)$, we have \begin{equation} \label{eqn:jumpcountingfunction} \delta \mathcal{N}_\delta(a_N)_{N \in \D}^\frac{1}{r} \leq V^r(a_N)_{N \in \D}. \end{equation} Indeed, for any $J \in \N$ for which we have times $N_0 \leq \dotsb \leq N_J$ with $|a_{N_{j+1}}-a_{N_j}| \geq \delta$ for all $j$ we have \[ V^r(a_N)_{N \in \D} \geq \left( \sum_{j=0}^{J-1} |a_{N_{j+1}}-a_{N_j}|^r \right)^\frac{1}{r} \geq \left( \sum_{j=0}^{J-1} \delta^r \right)^\frac{1}{r} = \delta J^\frac{1}{r}. \]

\begin{proposition}
Let $(X,\mu,T)$ be a measure preserving system and let $f \in L^{p_1}(X)$ and $g \in L^{p_2}(X)$ with $1<p_1,p_2<\infty$. Define $\frac{1}{p}=\frac{1}{p_1}+\frac{1}{p_2}$. Then \[ \sup_{\delta>0} \| \delta \mathcal{N}_\delta(A_{N;X}(f,g))_{N \in \D}^\frac{1}{r} \|_{L^p(X)} \leqsim_{p_1,p_2,r,\lambda} \|f\|_{L^{p_1}(X)} \|g\|_{L^{p_2}(X)} \] whenever $\D$ is $\lambda$-lacunary and $r>2$.
\end{proposition}

By specialising to the case where $X$ is the torus and $T$ is a rotation we obtain information on the rate of convergence of exponential sums.

\begin{corollary}
Given any $\lambda$-lacunary set $\D$ and any $\delta>0$, we have \[ \mathcal{N}_\delta \left( \frac{1}{N} \sum_{n=1}^N e(\zeta \lfloor \sqrt{n} \rfloor+\xi n) \right)_{N \in \D} \leqsim_{\varepsilon,\lambda} \delta^{-(2+\varepsilon)} \] for any $\varepsilon>0$. The implied constant is uniform in $\zeta,\xi \in \T$.
\end{corollary}

\begin{proof}
Consider the torus $\T$ with the normalised Lebesgue measure and define $T\:\T \to \T$ by $Tx=x+\alpha$ for some fixed $\alpha \in \T \setminus \{0\}$. Let $f(x) = e(\beta x)$ and $g(x)=e(\gamma x)$ for some fixed $\beta,\gamma \in \T$. Then \[ \|f\|_{L^2(\T)} = \|g\|_{L^2(\T)}=1 \] and \[ A_{N;\T}(f,g)(x) = e((\beta+\gamma)x) \frac{1}{N} \sum_{n=1}^N e(\alpha \beta \lfloor \sqrt{n} \rfloor+\alpha \gamma n), \] so \[ \|V^r(A_{N;\T}(f,g))_{N \in \D} \|_{L^1(\T)} = V^r \left( \frac{1}{N} \sum_{n=1}^N e(\alpha \beta \lfloor \sqrt{n} \rfloor+\alpha \gamma n) \right)_{N \in \D}. \] Taking $\alpha,\beta,\gamma \in \T$ such that $ \alpha \beta = \zeta$ and $\alpha \gamma = \xi$, part (iv) of Theorem \ref{thm:mainthm} implies \[ V^r\left( \frac{1}{N} \sum_{n=1}^N e(\zeta \lfloor \sqrt{n} \rfloor+\xi n) \right)_{N \in \D} \leqsim_{r,\lambda} 1 \] for any $r>2$, uniformly in $\zeta,\xi \in \T$.

\end{proof}

\section{Preliminaries} \label{sec:preliminaries}

\subsection{Notation}

We will use the convention that $\N = \{1,2,\dotsc\}$. For a real number $x$ we denote the greatest integer less than or equal to $x$ by $\lfloor x \rfloor$. For $x \in \R$ or $x \in \T$ we denote the complex exponential by $e(x) = e^{2\pi i x}$. All logarithms in this paper will be base 2.

Given a finite set $A$ and a function $f\:A \to \C$ we write \[ \E_{x \in A} f(x) = \frac{1}{|A|} \sum_{x \in A} f(x). \] Often, we will take $A = \{1, \dotsc, N\}$ or $A = \{-N, \dotsc, N\}$. We will denote these two sets by $[N]$ and $[[N]]$ respectively. More generally, if $N>0$ is a real number then we define $[N] = [ \lfloor N \rfloor ]$ and $[[N]] = [[ \lfloor N \rfloor]]$. We denote the indicator function of a set $A$ by $1_A$. Note in particular that if $A$ is a subset of a finite set $B$ then \[ \E_{x \in B} 1_A(x) = \frac{|A|}{|B|} \] denotes the density of $A$ in $B$.

We will work with the Fourier transform on $\Z$ primarily, but also on $\R$. The Fourier transform and inverse Fourier transform of a function $f\:\Z \to \C$  will be denoted by $\hat{f}$ and $\check{f}$ respectively. For functions on $\R$ we will always denote the Fourier transform and its inverse by $\F_\R$ and $\F_\R^{-1}$ respectively.

Recall that the Haar measure on $\Z$ is the counting measure, and the pontyagrin dual $\Z$ is $\T$ with the normalised Lebesgue measure. The Haar measure on $\R$ is the Lebesgue measure and the pontyagrin dual of $\R$ is itself.

\subsection{Normed and Quasi-Normed Spaces}

We will be dealing with various $\ell^p(\Z)$ norms throughout this paper. Although we assume in Theorem \ref{thm:reducedmainthm} that $f \in \ell^{p_1}(\Z)$ and $g \in \ell^{p_2}(\Z)$ with $1<p_1,p_2<\infty$, the condition $\frac{1}{p} = \frac{1}{p_1}+\frac{1}{p_2}$ does not ensure that $p \geq 1$. Therefore, while $f$ and $g$ will always be assumed to lie in a normed space, if $p<1$ then we will be viewing $A_N(f,g)$ as an element of $\ell^p(\Z)$ which is only a quasi-normed space. We outline the tools and differences we need to work in (quasi-)normed spaces.

Let $\{f_i\}_i$ be a countable collection of elements of $\ell^p(\Z)$. If $p \geq 1$ then one has the triangle inequality \[ \left\| \sum_i f_i \right\|_{\ell^p(\Z)} \leq \sum_i \|f_i\|_{\ell^p(\Z)}. \] However, if $p<1$ then we no longer have the triangle inequality. Instead we have the \emph{quasi-triangle inequality} \begin{equation} \label{eqn:quasitriangleinequality} \left\| \sum_i f_i \right\|_{\ell^p(\Z)}^p \leq \sum_i \|f_i\|_{\ell^p(\Z)}^p.
\end{equation} A partial resolution is to note that if $k \in \N$ and $f_1, \dotsc, f_k \in \ell^p(\Z)$ then by (\ref{eqn:quasitriangleinequality}) and Hölder's inequality we have \begin{equation} \label{eqn:alternativequasi} \left\| \sum_{i=1}^k f_i \right\|_{\ell^p(\Z)} \leq k^{\frac{1}{p}-1} \sum_{i=1}^k \|f_i\|_{\ell^p(\Z)}. \end{equation} Thus if $k = O(1)$, for example, we are free to replace the triangle inequality with (\ref{eqn:alternativequasi}) freely.

A simple but useful consequence of the triangle inequality in $\ell^p(\Z)$ is that $A_N$ is bounded as an operator from $\ell^{p_1}(\Z) \times \ell^{p_2}(\Z)$ to $\ell^p(\Z)$ whenever $\frac{1}{p_1}+\frac{1}{p_2}=\frac{1}{p} \leq 1$. Indeed, by applying the triangle inequality in $\ell^p(\Z)$, Hölder's inequality, and translation invariance we obtain \begin{equation} \label{eqn:HolderBanach} \|A_N(f,g)\|_{\ell^p(\Z)} \leq \|f\|_{\ell^{p_1}(\Z)} \|g\|_{\ell^{p_2}(\Z)}. \end{equation} Since we will be dealing with cases where $p<1$, we would like an analogue of (\ref{eqn:HolderBanach}) 

\begin{lemma} \label{lem:holdernonbanach}
Let $1<p_1,p_2<\infty$ and define $\frac{1}{p} = \frac{1}{p_1}+\frac{1}{p_2}$. Then for any measure preserving system $(X,\mu,T)$ we have \begin{equation} \label{eqn:holdernonbanach} \|A_{N;X}(f,g)\|_{L^p(X)} \leqsim_{p_1,p_2} \|f\|_{L^{p_1}(X)} \|g\|_{L^{p_2}(X)} \end{equation} for all $N\geq 1$, $f \in L^{p_1}(X)$, and $g \in L^{p_2}(X)$. The same holds with $A_{N;X}$ replaced by $\Tilde{A}_{N;X}$.
\end{lemma}

\begin{proof}
By the Calderón Transference Principle it suffices to prove (\ref{eqn:holdernonbanach}) in the integer case. Applying the triangle inequality and utilising non-negativity, we may bound \begin{equation} \label{eqn:reducetointervals} |A_N(f,g)(x)| \leq \sum_{I \in \mathcal{I}} 1_I(x) A_N(|f|1_I, |g|1_I)(x) \end{equation} where $\mathcal{I}$ is a cover of $\Z$ by intervals of length $O(N)$ which overlap $O(1)$. Assume that $f$ and $g$ are supported on one such interval $I$. Then $A_N(f,g)$ is supported on an interval of length $O(N)$ so by Hölder's inequality we have \[ \|A_N(f,g)\|_{\ell^p(\Z)} \leqsim N^{\frac{1}{p}-1} \|A_N(f,g)\|_{\ell^1(\Z)}. \] By the triangle inequality, Fubini's theorem, and translation invariance we can bound \[ \| A_N(f,g)\|_{\ell^1(\Z)} \leq \left\| |f| \cdot \frac{1}{N} \sum_{n=1}^N |g(\cdot+\lfloor \sqrt{n} \rfloor-n)| \right\|_{\ell^1(\Z)} = \| |f| \cdot B_N(|g|) \|_{\ell^1(\Z)} \] where \[ B_Ng(x) = \frac{1}{N} \sum_{n=1}^N g(x+\lfloor \sqrt{n} \rfloor-n). \] Then Hölder's inequality again and Proposition \ref{prop:improving} (see Section \ref{sec:majorarc}) yield \[ \|A_N(f,g)\|_{\ell^1(\Z)} \leq \|f\|_{\ell^{p_1}(\Z)} \|B_Ng\|_{\ell^{p_1'}(\Z)} \leqsim N^{-(\frac{1}{p}-1)} \|f\|_{\ell^{p_1}(\Z)} \|g\|_{\ell^{p_2}(\Z)} \] hence we obtain (\ref{eqn:holdernonbanach}) in the case where $f$ and $g$ are supported on $I$. The result follows by applying the quasi-triangle inequality to (\ref{eqn:reducetointervals}) and using Hölder's inequality.

\end{proof}

\subsection{Littlewood--Paley Theory}

Throughout this paper we fix a cutoff function $\Psi \in C_c^\infty(\R;[0,1])$ which is non-negative, even, supported on $[-1,1]$, and equal to 1 on $[-\frac{1}{2},\frac{1}{2}]$. We allow all implied constants to depend on $\Psi$. For a real number $x$, we define \[ \Psi_{\leq x} = \Psi\left(\frac{\cdot}{2^{\lceil \log x \rceil}}\right), \]  \[ \Psi_x = \Psi_{\leq x} - \Psi_{\leq \frac{x}{2}}, \] and \[ \Psi_{>x} = 1-\Psi_{\leq x}. \]Then $\Psi_{\leq x}$ is supported on $\{ |\xi| \leq 2^{\lceil \log x \rceil} \}$ with $2^{\lceil \log x \rceil} \sim x$. Similarly $\Psi_x$ is supported on $\{ 2^{\lceil \log x \rceil-1}< |\xi| \leq 2^{\lceil \log x \rceil} \}$. We denote the supports of $\Psi_{\leq x}$ and $\Psi_x$ by $\M_{\leq x}$ and $\M_x$ respectively.

We will often localise a function $f\:\Z \to \C$ in frequency space to $\{|\xi| \leq 2^{-j} \}$ by convolving $f$ with $\widecheck{\Psi}_{\leq 2^{-j}}$. We note that the operator $f \mapsto \widecheck{\Psi}_{\leq 2^{-j}}*f$ is bounded on $\ell^p(\Z)$ for any $p \geq 1$. Indeed, the triangle inequality implies $\|\widecheck{\Psi}_{\leq 2^{-j}}*f\|_{\ell^p(\Z)} \leq \|\widecheck{\Psi}_{\leq 2^{-j}} \|_{\ell^1(\Z)} \|f\|_{\ell^p(\Z)}$ so it suffices to show that $\|\widecheck{\Psi}_{\leq 2^{-j}} \|_{\ell^1(\Z)} \leqsim 1$. But this is easily done by splitting the domain of summation into $|y| \leq 2^j$ and $|y|>2^j$. By writing $\Psi_{2^{-j}} = \Psi_{\leq 2^{-j}} - \Psi_{\leq 2^{-j-1}}$ and $\Psi_{>2^{-j}} = 1-\Psi_{\leq 2^{-j}}$ and applying the triangle inequality we see that \begin{equation} \|\widecheck{\Psi}_{\leq 2^{-j}}*f \|_{\ell^p(\Z)}, \|\widecheck{\Psi}_{2^{-j}} *f \|_{\ell^p(\Z)}, \|\widecheck{\Psi}_{> 2^{-j}} *f\|_{\ell^p(\Z)} \leqsim \|f\|_{\ell^p(\Z)} \label{eqn:psiboundednessZ} \end{equation} for all $f \in \ell^p(\Z)$ and $p \geq 1$.

\subsection{Variation Norms}

We will use various standard facts about variation norms throughout the course of the paper. Let $\D$ be any countable ordered set and let $a,b\:\D \to \C$ be sequences of complex numbers. In this paper the ordered sets will be subsets of the natural numbers. As discussed previously, for $0<r <\infty$ the $V^r$ norm of the sequence $(a_N)$ is defined as \[ V^r(a_N)_{N \in \D} = \sup_{N \in \D} |a_N| + \sup_{J \in \N} \sup_{\substack{N_0<\dotsb<N_J \\ N_j \in \D}} \left( \sum_{j=0}^{J-1} |a_{N_{j+1}}-a_{N_j}|^r \right)^\frac{1}{r} \] and \[ V^\infty(a_N)_{N \in \D} = \sup_{N \in \D} |a_N| + \sup_{\substack{M<N \\ M,N \in \D}} |a_M-a_N|. \] Then $V^r$ is a norm on the space of functions from $\D$ to $\C$.

The first fact we will commonly use is that one can bound a $V^r$ norm by an $\ell^r$ norm. Since $\D$ is countable one can then bound an $\ell^r$ norm above by an $\ell^1$ norm. Thus we have
\begin{equation} V^r(a_N)_{N \in \D} \leqsim \left( \sum_{N \in \D} |a_N|^r \right)^\frac{1}{r} \leq \sum_{N \in \D} |a_N|. \label{eqn:variationvsellr} \end{equation}

We now look at the $V^r$ norm of a product. This will be important because we would like to make use of Hölder's inequality. To do so, we will need to split the $V^r$ norm of a product into a product of $V^r$ norms. Indeed, we have
\begin{equation}
\label{eqn:variationofproduct} V^r(a_N b_N)_{N \in \D} \leqsim V^r(a_N)_{N \in \D} V^r(b_N)_{N \in \D}.
\end{equation}

In order to obtain lower bounds on the elements of the indexing set $\D$, the following fact is useful. One has \begin{equation} \label{eqn:orderedpartition} V^r(a_N)_{N \in \D} \leqsim V^r(a_N)_{N \in \D_1} + V^r(a_N)_{N \in \D_2} \end{equation} whenever $\D_1$ and $\D_2$ form a partition of $\D$ such that $N_1<N_2$ whenever $N_1 \in \D_1$ and $N_2 \in \D_2$.

Finally, we note that the $V^\infty$ norm is comparable to the $\ell^\infty$ norm: \begin{equation} \label{eqn:Vinfinity} V^\infty(a_N)_{N \in \D}  \sim \| a\|_{\ell^\infty(\D)}. \end{equation}

\subsection{Gowers Norms}

A common tool used in additive combinatorics is a class of norms called \emph{Gowers norms}.

Given a function $f\:\Z \to \C$ and $h \in \Z$, we define the \emph{differencing} operator $\Delta_h$ by \[ \Delta_h f(x) = f(x) \overline{f(x+h)}. \] Then given $h_1, \dotsc, h_s \in \Z$ we inductively define \begin{equation} \Delta_{h_1, \dotsc, h_s} f(x) = \Delta_{h_1} \dotsm \Delta_{h_s} f(x). \label{eqn:differencingoperator} \end{equation} Note that this definition is independent of any permutations of $h_1, \dotsc, h_s$. Concretely, if $s=2$ then we have \[ \Delta_{h_1,h_2}f(x) = f(x) \overline{f(x+h_1) f(x+h_2)} f(x+h_1+h_2) \] and if $s=3$ we have \begin{multline*} \Delta_{h_1,h_2,h_3} f(x) = \\ f(x) \overline{f(x+h_1) f(x+h_2) f(x+h_1)} f(x+h_1+h_2) f(x+h_1+h_3) f(x+h_2+h_3) \overline{f(x+h_1+h_2+h_3)}. \end{multline*} More generally, if $\sigma \:\C \to \C$ denotes complex conjugation, then \[ \Delta_{h_1, \dotsc, h_s} f(x) = \prod_{\omega_1, \dotsc, \omega_s \in \{0,1\} } \sigma^{\omega_1+\dotsb+\omega_s} f(x+\omega_1 h_1+\dotsb+\omega_s h_s). \] If $h=(h_1, \dotsc, h_s)$ we may denote $\Delta_{h_1, \dotsc, h_s}$ by $\Delta_h$ for convenience.

With this definition in hand we can define the Gowers norms. Let $s \in \N$ and suppose $f\:\Z \to \C$ is supported on a finite set. The unnormalised or global (\emph{Gowers}) \emph{$U^s$ norm} of $f$ is defined by \begin{equation} \label{eqn:Usdefinition} \|f\|_{U^s(\Z)}^{2^s} = \sum_{x,h_1, \dotsc, h_s \in \Z} \Delta_{h_1, \dotsc, h_s} f(x). \end{equation} Note by (\ref{eqn:Usdefinition}) and (\ref{eqn:differencingoperator}) that if $k \leq s$ then \begin{equation} \|f\|_{U^s(\Z)}^{2^s} = \sum_{h_1, \dotsc, h_{s-k} \in \Z} \| \Delta_{h_1, \dotsc, h_{s-k}} f \|_{U^k(\Z)}^{2^k}. \label{eqn:Gowersintermsofgowers} \end{equation} The $U^1$ norm is in fact only a seminorm, not a norm, but for $s \geq 2$ the $U^s$ norm is a genuine norm.

The $U^2$ norm is particularly useful because of the Fourier analytic relation \begin{equation} \label{eqn:U2characterisation}
\|f\|_{U^2(\Z)} = \| \hat{f} \|_{L^4(\T)}.
\end{equation} This connection between the $U^2$ norm and the Fourier transform allows us to use the following tool, which can be found in \cite{PelusePrendivilleInverseNonLinearRoth}, for example.

\begin{lemma} \label{lem:U2inverse}
Let $f\:\Z \to \C$ be a $1$-bounded function with support with finite support. Then \[ \|f\|_{U^2(\Z)}^4 \leq |\supp f| \left| \sum_{x \in \Z} f(x)e(x\xi) \right|^2 \] for some $\xi \in \T$.
\end{lemma}

\begin{proof}
Recall from (\ref{eqn:U2characterisation}) that $\|f\|_{U^2(\Z)} = \|\hat{f} \|_{L^4(\T)}$. Thus by Plancherel's theorem we have \[ \|f\|_{U^2(\Z)}^4 = \|\hat{f} \|_{L^4(\T)}^4 \leq \|\hat{f} \|_{L^\infty(\T)}^2 \|\hat{f} \|_{L^2(\T)}^2 = \| \hat{f} \|_{L^\infty(\T)}^2 \|f \|_{\ell^2(\Z)}^2 \leq |\supp f| \cdot \|\hat{f} \|_{L^\infty(\T)}^2, \] and the result follows.
\end{proof}

\subsection{Shifted Square Functions}

The following theorem on $\ell^p$ bounds of shifted square functions was used in \cite{KMT} (with $\Z$ and $\T$ both replaced by $\R$). We present the integer formulation here and briefly discuss its proof.

\begin{theorem} \label{thm:shifted}
Let $\D$ be a finite $\lambda$-lacunary set for some $\lambda>1$, and let $A>0$, $0<C<\frac{1}{2}$, $d>1$, and $K>1$. Fix $\eta \in \S(\T)$ which vanishes at the origin and suppose $\eta$ is supported on $[-C,C]$ and satisfies \[ \| \partial^{(j)} \eta\|_{L^\infty(\T)} \leq C \] for each $j \in \{0,1,2\}$. For each $N \in \D$ let $\lambda_N \in [-K,K]$ and let $\eta_N=\eta(AN^d \xi)$. Then for any $1<p<\infty$ we have \[ \left\| \left( \sum_{N \in \D} |\F_\R^{-1} \eta_N(\cdot-\lambda_N A N^d) *f|^2 \right)^\frac{1}{2} \right\|_{\ell^p(\Z)} \leqsim_{C,\lambda,d,p} \log K \|f\|_{\ell^p(\Z)} \] for all $f \in \ell^p(\Z)$.
\end{theorem}

\begin{proof}
By Khintchine's inequality it suffices to prove that \[ \left\| \sum_{N \in \D} \varepsilon_N \F_\R^{-1}\eta_N(\cdot-\lambda_NAN^d)*f \right\|_{\ell^p(\Z)} \leqsim_{C,\lambda,d,p} 
\log K \|f\|_{\ell^p(\Z)} \] for any $1$-bounded complex numbers $\varepsilon_N$ for $N \in \D$.

Let $\Tilde{\eta} = \sum_{N \in \D} \varepsilon_N \eta_N$. The assumption that $\eta$ vanishes at the origin and is supported on $|\xi| \leq C$ yields $\|\Tilde{\eta}\|_{L^\infty(\T)} \leqsim_{C,\lambda} 1$. Thus Plancherel's theorem yields the result for $p=2$. By duality and the Marcinkiewicz interpolation theorem it then suffices to prove the weak-$(1,1)$ inequality \[ |\{x \in \Z \mid |(\F_\R^{-1}\Tilde{\eta}*f)(x)| \geq \alpha \}| \leqsim_{C,\lambda} \frac{\log K\|f\|_{\ell^1(\Z)}}{\alpha} \] for all $\alpha>0$. One can perform a Calderón--Zygmund decomposition on the integers, and the usual proof that Calderón--Zygmund operators are weak-$(1,1)$ reduces to showing that \[ \sum_{x \in (100Q_j)^c} |(\F_\R^{-1} \Tilde{\eta}* b_j)(x)| \leqsim_{C,\lambda} \log K \|f\|_{\ell^1(Q_j)} \] where $b_j=b1_{Q_j}$, $b$ is the ``bad part" of $f$, and $\{Q_j\}_{j=1}^\infty$ are a disjoint collection of dyadic blocks whose union is \[ \{x \in \Z \mid |(\F_\R^{-1}\Tilde{\eta}* f)(x)| \geq \alpha \}. \] By writing \[ \F_\R^{-1} \Tilde{\eta} = \sum_{N \in \D} \varepsilon_N \sum_{y \in \Z} \frac{1}{AN^d} \F_\R^{-1} \eta\left(\frac{\cdot}{AN^d} \right) \] and using the fact that $b_j$ has mean zero, we may bound \begin{multline*} \sum_{x \in (100Q_j)^c} |(\F_\R^{-1} \Tilde{\eta}* b_j)(x)| \\ \leq \sum_{x \in (100Q_j)^c} \sum_{N \in \D} \sum_{y \in Q_j} \frac{1}{AN^d} \left|\F_\R^{-1} \eta\left(\frac{x-y-\lambda_N AN^d}{AN^d} \right) - \F_\R^{-1} \eta\left(\frac{x-\lambda_N AN^d}{AN^d} \right) \right| |b_j(y)|. \end{multline*} At this point the argument is standard and can be found in Chapter 1 of \cite{stein1993harmonic}.
\end{proof}

\section{Additive Combinatorial Inverse Theorem} \label{sec:inversethm}

We would like to prove a minor arc estimate which will allow us to reduce to the case where $f$ and $g$ are Fourier supported on major arcs.

To control the minor arc contributions, we will need an analogue of Theorem \ref{thm:peluseinverse} for our operator which informally says that if $\Tilde{A}_N(f,g)$ is ``larger" than some threshold $\delta$, then $f_1$ has a Fourier transform with a ``large" contribution on the principal major arc $\M_{\leq \delta^{-O(1)} N^{-\frac{1}{2}}}$. This section is dedicated to proving such an analogue; the minor arc estimate will be handled in the next section.

The goal of this section is to prove the following theorem.

\begin{theorem} \label{thm:inversethm}
Let $N \geq 1$ and $0<\delta \leq 1$. Let $f_0,f_1,f_2\:\Z \to \C$ be $1$-bounded functions supported on $[[N_0]]$ for some $N_0 \sim N$ and suppose \begin{equation} \label{eqn:largeaverage} \left| \sum_{x \in \Z} f_0(x)  \Tilde{A}_N(f_1,f_2)(x) \right| \geq \delta N. \end{equation} Then either $N \leqsim \delta^{-O(1)}$ or there exists $N' \in \N$ with $\delta^{O(1)} N^\frac{1}{2} \leqsim N' \leq N^\frac{1}{2}$ such that \[ \sum_{x \in \Z} |\E_{n \in [N']} f_1(x-n)| \geqsim \delta^{O(1)}N. \]
\end{theorem}

Note that the Fourier transform of $\E_{n \in [N']} f_1(x-n)$ is \[ \hat{f}_1(\xi) \E_{n \in [N']} e(-n\xi) = O(\delta^{-O(1)} N^{-\frac{1}{2}} \|\xi\|_\T^{-1}) \hat{f}_1(\xi), \] which is large when $\xi \in \M_{\leq \delta^{-O(1)} N^{-\frac{1}{2}}}$, formalising the intuition discussed above.

As discussed above, we will manipulate (\ref{eqn:largeaverage}) by a change of variables into a form where one can use Theorem \ref{thm:peluseinverse} directly. However, since our operator only has a large contribution on the principal major arc, we will use Theorem 5.6 of \cite{PelusePrendivilleInverseNonLinearRoth}, which is one of the earlier results used to prove Theorem \ref{thm:peluseinverse} (at least in the case where $P(n)=n^2$).

\begin{theorem}[Peluse, Prendiville] \label{thm:peluseprendiville}
Let $N \geq 1$ and $0<\delta \leq 1$. Let $f_0,f_1,f_2\:\Z \to \C$ be $1$-bounded functions supported on $[[N_0]]$ for some $N_0 \sim N$ and suppose \begin{equation} \left| \sum_{x \in \Z} f_0(x) \cdot \frac{1}{N} \sum_{n=1}^N f_1(x-n) f_2(x-n^2) \right| \geq \delta N^2. \end{equation} Then either $N \leqsim \delta^{-O(1)}$ or $\|f_2\|_{U^5(\Z)}^{2^5} \geqsim \delta^{O(1)} N^6$.
\end{theorem}

We will then modify the proof used in \cite{PelusePrendivilleInverseNonLinearRoth} in order to obtain principal major arc behaviour in $f_1$ instead of just major arc behaviour.

\subsection{Gowers Norms Control Averages}

The first step in the proof of Theorem \ref{thm:inversethm} is to show that Gowers norms control the averages $\Tilde{A}_N(f_1,f_2)$. More precisely, we show that under the assumptions of Theorem \ref{thm:inversethm} the $U^s$ norm of $f_2$ must be large for some $s$.

\begin{proposition} \label{prop:U5inverse}
Under the hypotheses of Theorem \ref{thm:inversethm} either $N \leqsim \delta^{-O(1)}$ or $\|f_2\|_{U^5(\Z)}^{2^5} \geqsim \delta^{O(1)} N^6$.
\end{proposition}

\begin{proof}
By assumption we have \[ \delta N \leq \left| \sum_{x \in \Z} f_0(x) \cdot \frac{1}{N} \sum_{n=1}^N 1_{n>\frac{N}{2}} f_1(x-\lfloor \sqrt{n} \rfloor) f_2(x-n) \right|. \] We split $[N]$ into intervals on which $\lfloor \sqrt{n} \rfloor$ is constant. Thus \[ \delta N \leq \left| \sum_{x \in \Z} f_0(x) \cdot \frac{1}{N} \sum_{m \leq \sqrt{N}} \sum_{m^2 \leq n <(m+1)^2} 1_{n>\frac{N}{2}} f_1(x-m) f_2(x-n) \right| + O (N^\frac{1}{2}) \] so as long as $N \geqsim \delta^{-O(1)}$ we have \[ \delta N \leqsim \left| \sum_{x \in \Z} f_0(x) \cdot \frac{1}{N} \sum_{m \leq \sqrt{N}} \sum_{m^2 \leq n<(m+1)^2} 1_{n>\frac{N}{2}} f_1(x-m) f_2(x-n) \right|. \] We change variables in $n$ to obtain \[ \delta N \leqsim \left| \sum_{x \in \Z} f_0(x) \cdot \frac{1}{N} \sum_{m \leq \sqrt{N}} \sum_{n=0}^{2m} 1_{n+m^2>\frac{N}{2}} f_1(x-m) f_2(x-n-m^2) \right|, \] and pull out the sum in $n$ and apply the triangle inequality so that \[ \delta N \leqsim \frac{1}{\sqrt{N}} \sum_{n \leq 2\sqrt{N}} \left| \sum_{x \in \Z} f_0(x) \cdot \frac{1}{\sqrt{N}} \sum_{m \leq \sqrt{N}} 1_{n+m^2>\frac{N}{2}} 1_{0 \leq n \leq 2m} f_1(x-m) f_2(x-n-m^2) \right|. \] By the pigeonhole principle there exists $n$ such that \[ \delta N \leqsim \left| \sum_{x \in \Z} f_0(x) \cdot \frac{1}{\sqrt{N}} \sum_{m \leq \sqrt{N}} 1_{m \geq \max(\sqrt{\frac{N}{2}-n},\frac{n}{2})} f_1(x-m) f_2(x-n-m^2) \right|. \] Note that for all $N \geqsim 1$ we have $\max(\sqrt{\frac{N}{2}-n},\frac{n}{2}) \sim \sqrt{N}$. Thus by the triangle inequality and the pigeonhole principle we have \begin{equation} \delta M^2 \leqsim \left| \sum_{x \in \Z} f_0(x) \cdot \frac{1}{M} \sum_{m=1}^M f_1(x-m) f_2(x-n-m^2) \right| \label{eqn:U5changeofvariables} \end{equation} for some $M \sim \sqrt{N}$ (which may depend on $n$). Since $f_0,f_1,f_2(\cdot-n)$ are $1$-bounded functions supported on $[[M_0]]$ with $M_0 \sim M^2$, we may apply Theorem \ref{thm:peluseprendiville} to obtain $\|f_2(\cdot-n) \|_{U^5(\Z)}^{2^5} \geqsim \delta^{O(1)} (M^2)^6 \sim \delta^{O(1)} N^6$. But Gowers norms are translation invariant so $\|f_2\|_{U^5(\Z)}^{2^5} \geqsim \delta^{O(1)} N^6$.
\end{proof}

\subsection{Gowers Norms Control the Dual Function}

The conclusion of Proposition \ref{prop:U5inverse} does not contain any information about the sequences $\lfloor \sqrt{n} \rfloor$ and $n$ along which we are sampling our functions. The solution to this is to show that under the same hypotheses, a \emph{dual function} associated to the trilinear form \[ (f_0,f_1,f_2) \mapsto \sum_{x \in \Z} f_0(x) \Tilde{A}_N(f_1,f_2)(x) \] (which does contain information about these sequences) also has a large $U^5$ norm. The dual function is defined as \[ \Tilde{A}_N^{**}(f_0,f_1)(x) = \frac{1}{N} \sum_{n=1}^N 1_{n>\frac{N}{2}} f_0(x+n) f_1(x-\lfloor \sqrt{n} \rfloor+n). \] Note that \[ \sum_{x \in \Z} f_0(x) \Tilde{A}_N(f_1,f_2)(x) = \sum_{x \in \Z} f_2(x) \Tilde{A}_N^{**}(f_0,f_1)(x). \] The choice of notation is explained by the fact that there are two associated dual functions but only this one will relevant for now. We will see the other in the next section.

\begin{proposition} \label{prop:dualinverse}
Under the hypotheses of Theorem \ref{thm:inversethm} one also has $\|\Tilde{A}_N^{**}(f_0,f_1) \|_{U^5(\Z)}^{2^5} \geqsim \delta^{O(1)} N^6$.
\end{proposition}

\begin{proof}
By assumption we have \[ \delta N \leq \left| \sum_{x \in \Z} f_0(x) \cdot \frac{1}{N} \sum_{n=1}^N 1_{n>\frac{N}{2}} f_1(x-\lfloor \sqrt{n} \rfloor) f_2(x-n) \right|. \] Making the change of variables $x \mapsto x+n$ yields \[ \delta N \leq \left|\sum_{x \in \Z} f_2(x) \cdot \frac{1}{N} \sum_{n=1}^N 1_{n>\frac{N}{2}} f_0(x+n) f_1(x-\lfloor \sqrt{n} \rfloor+n) \right| \] hence the Cauchy--Schwarz inequality implies \[ \delta^2 N \leqsim \sum_{x \in \Z} \left| \frac{1}{N} \sum_{n=1}^N 1_{n>\frac{N}{2}} f_0(x+n) f_1(x-\lfloor \sqrt{n} \rfloor+n) \right|^2. \] Expanding out the square, we have \[ \delta^2 N \leqsim \sum_{x \in \Z} \frac{1}{N} \sum_{n=1}^N 1_{n>\frac{N}{2}} f_0(x+n) f_1(x-\lfloor \sqrt{n} \rfloor+n) \overline{\frac{1}{N} \sum_{m=1}^N 1_{m>\frac{N}{2}} f_0(x+m)f_1(x-\lfloor \sqrt{m} \rfloor+m)}. \] Finally, making the change of variables $x \mapsto x-n$ yields \[ \delta^2 N \leqsim \sum_{x \in \Z} f_0(x) \cdot \frac{1}{N} \sum_{n=1}^N 1_{n>\frac{N}{2}} f_1(x-\lfloor \sqrt{n} \rfloor) \overline{\Tilde{A}_N^{**}(f_0,f_1)(x-n)}, \] or in other words \[ \delta^2 N \leqsim \left| \sum_{x \in \Z} f_0(x) \Tilde{A}_N(f_1, \overline{\Tilde{A}_N^{**}(f_0,f_1)})(x) \right|. \] Note that $\overline{\Tilde{A}_N^{**}(f_0,f_1)}$ is a $1$-bounded function supported on $[N_0']$ with $N_0' \sim N$. Thus we can apply Proposition \ref{prop:U5inverse} and the invariance of Gowers norms under complex conjugation to deduce $\|\Tilde{A}_N^{**}(f_0,f_1)\|_{U^5(\Z)}^{2^5} \geqsim \delta^{O(1)}N^6$.
\end{proof}

\subsection{Degree Lowering}

We now have $U^5$ control of the dual function $\Tilde{A}_N^{**}(f_0,f_1)$. However, due to the connections between the $U^2$ norm and the Fourier transform (as demonstrated by (\ref{eqn:U2characterisation}) and Lemma \ref{lem:U2inverse}) we would like to have $U^2$ control. The point of this subsection is to show that $U^s$ control of the dual function implies $U^{s-1}$ control, a result known as degree lowering.

We will need the following two lemmas, which can be found in \cite{PelusePrendivilleInverseNonLinearRoth}. The first allows us to ``swap" differencing operators and dual functions.

\begin{lemma}[Dual--Difference Interchange] \label{lem:dualdifferenceinterchange}
For each $n \in [N]$ let $F_n\:\Z \to \C$ be a $1$-bounded function supported on $[[N_0]]$ with  $N_0 \sim N$. Set \[ F(x)= \E_{n \in [N]} F_n(x).\] Then for any function $\xi \:\Z^s \to \T$ and finite set $X \subseteq \Z^s$ we have \begin{multline*} \left(\frac{1}{N^{s+1}} \sum_{h \in X} \left| \sum_{x \in \Z} \Delta_h F(x) e(\xi(h)x) \right| \right)^{2^s} \\ \leqsim_s \frac{1}{N^{2s+1}} \sum_{h^{(0)},h^{(1)} \in X} \left| \sum_{x \in \Z} \E_{n \in [N]} \Delta_{h^{(0)}-h^{(1)}} F_n(x) e(\xi(h^{(0)};h^{(1)})x) \right|, \end{multline*} where \[ \xi(h^{(0)};h^{(1)}) = \sum_{\omega \in \{0,1\}^s} (-1)^{|\omega|} \xi(h^{(\omega)}) \] and $h^{(\omega)} = (h^{(\omega_1)}_1, \dotsc, h^{(\omega_s)}_s)$.
\end{lemma}

The next lemma will require some terminology. If $\phi\:\Z^{s-1} \to \C$ is a function of $s-1$ variables and $h_1, \dotsc, h_s \in \Z$ we write \[ \phi(h_1, \dotsc, h_{i-1}, \hat{h}_i, h_{i+1}, \dotsc, h_s) = \phi(h_1, \dotsc, h_{i-1},h_{i+1}, \dotsc, h_s). \] We now say that a function $\phi\:\Z^s \to \C$ is \emph{low rank} if there exist functions $\phi_1, \dotsc, \phi_s\:\Z^{s-1} \to \C$ such that \begin{equation} \phi(h_1, \dotsc, h_s) = \sum_{i=1}^s \phi_i(h_1, \dotsc, \hat{h}_i, \dotsc, h_s). \end{equation}

It can be shown by (\ref{eqn:Gowersintermsofgowers}), the popularity principle, and Lemma \ref{lem:U2inverse} that if $f$ is supported on $[[N_0]]$ with $N_0 \sim N$ and $\|f\|_{U^{s+2}}^{2^{s+2}} \geq \delta N^{s+3}$ then \[ \delta N^{s+1} \leqsim \sum_{h_1, \dotsc, h_s \in \Z} \left| \sum_{x \in \Z} \Delta_{h_1, \dotsc, h_s} f(x) e(\xi(h_1, \dotsc, h_s)x) \right| \] for some function $\xi\:\Z^s \to \C$. The next lemma says that if $\xi$ is low rank then we can bound the right hand side in terms of $\|f\|_{U^{s+1}}$.

\begin{lemma} \label{lem:lowrank}
Let $f\:\Z \to \C$ be a $1$-bounded function supported on $[[N_0]]$ with $N_0 \sim N$. For functions $\xi_1, \dotsc, \xi_m\:\Z^{s-1} \to \T$ with $m \leq s$ we have \[ \frac{1}{N^{s+1}} \sum_{h_1, \dotsc, h_s \in \Z} \left| \sum_{x \in \Z} \Delta_{h_1, \dotsc, h_s} f(x) e \left( \sum_{i=1}^m \phi_i(h_1, \dotsc, \hat{h}_i, \dotsc, h_s)x \right) \right| \leqsim_s \left( \frac{\|f\|_{U^{s+1}}^{2^{s+1}}}{N^{s+2}} \right)^\frac{1}{2^{m+1}} \]
\end{lemma}

The final tool needed is an $L^\infty$ type bound on the exponential sum \[ m_{N;\Z}(\zeta,\xi) = \frac{1}{N} \sum_{n=1}^N 1_{n>\frac{N}{2}} e(\zeta \lfloor \sqrt{n} \rfloor + \xi n). \] In \cite{PelusePrendivilleInverseNonLinearRoth} the precise bound used is that if \[ \delta \leq \left| \frac{1}{N} \sum_{n=1}^N e(\zeta n + \xi n^2) \right| \] then either $N \leqsim \delta^{-O(1)}$ or there exists $q \leqsim \delta^{-O(1)}$ such that $\|q \xi \|_\T \leqsim \delta^{-O(1)} N^{-2}$. We will see shortly that if $|m_{N;\Z}(\zeta,\xi)| \geq \delta$ then a change of variables yields \[ \delta \leqsim \left| \frac{1}{M} \sum_{m=1}^M e(\zeta m+\xi m^2) \right| \] where $\delta \sqrt{N} \leqsim M \leqsim \sqrt{N}$. As before, we could at this point cite the same result that Peluse and Prendiville used, but as usual we are aiming to obtain principal major arc behaviour so we will prove a stronger result.

\begin{lemma} \label{lem:exponentialsum}
If \[ \delta \leq \left| \frac{1}{N} \sum_{n=1}^N 1_{n>\frac{N}{2}} e(\zeta \lfloor \sqrt{n} \rfloor + \xi n) \right| \] then either $N \leqsim \delta^{-O(1)}$ or $\|\xi\|_\T \leqsim \delta^{-O(1)} N^{-1}$.
\end{lemma}

\begin{proof}
By the triangle inequality and the pigeonhole principle we have \[ \delta \leqsim \left| \frac{1}{N} \sum_{n=1}^{N'} e(\zeta \lfloor \sqrt{n} \rfloor + \xi n) \right| \] with $N' \in \{\frac{N}{2},N \}$. We split the interval $[N']$ into intervals on which $\lfloor \sqrt{n} \rfloor$ is constant, so that \[ \frac{1}{N} \sum_{n=1}^{N'} e(\zeta \lfloor \sqrt{n} \rfloor + \xi n) = \frac{1}{N} \sum_{m \leq \sqrt{N'}} \sum_{m^2 \leq n<(m+1)^2} e(\zeta m + \xi n) + O\left(\frac{1}{\sqrt{N}} \right). \] Thus if $N \geqsim \delta^{-O(1)}$ we have \[ \delta \leqsim \left|\frac{1}{N} \sum_{m \leq \sqrt{N'}} e(\zeta m) \sum_{m^2 \leq n <(m+1)^2} e(\xi n) \right|, \] and by changing variables in $n$ we may rewrite this as \[ \delta \leqsim \left| \frac{1}{N} \sum_{m \leq \sqrt{N'}} e(\zeta m+\xi m^2) \sum_{n=0}^{2m} e(\xi n) \right|. \] Let \[ D_k(\xi) = \sum_{j=0}^{k-1} e(\xi j) = e\left(\frac{k-1}{2} \right) \frac{\sin(k \pi \xi)}{\sin(\pi \xi)} \] denote the unnormalised Dirichlet kernel. Then we have \[ \delta \leqsim \left| \frac{1}{\sqrt{N}} \sum_{m \leq \sqrt{N'}} e(\xi m^2+\zeta m) \cdot \frac{D_{2m+1}(\xi)
}{\sqrt{N}} \right|. \] For each $m \in [\sqrt{N'}]$ let \[ \psi_\xi(m) = \frac{D_{2m+1}(\xi)
}{\sqrt{N}} \] and notice that $\psi_\xi$ satisfies the bound \[ \|\psi_\xi\|_{\ell^\infty([\sqrt{N'}])} + \|\psi_\xi(\cdot+1)-\psi_\xi \|_{\ell^1([\sqrt{N'}])} \leqsim 1. \] Moreover, for any interval $I \subseteq [\sqrt{N'}]$ we have \[ \left| \frac{1}{\sqrt{N}} \sum_{m \in I} e(\xi m^2+\zeta m) \right| \leqsim \sup_{M \leq \sqrt{N'}} \left| \frac{1}{\sqrt{N}} \sum_{m=1}^M e(\xi m^2+\zeta m) \right|. \] Thus summation by parts yields \[ \delta \leqsim \sup_{M \leq \sqrt{N'}} \left| \frac{1}{\sqrt{N}} \sum_{m=1}^M e(\xi m^2+\zeta m) \right| \] and so there exists $M \leq \sqrt{N'}$ such that \[ \delta \leqsim \left| \frac{1}{\sqrt{N}} \sum_{m=1}^M e(\xi m^2+ \zeta m) \right|. \] By the triangle inequality, trivially bounding the exponential by 1, we see that in fact $M$ satisfies $\delta \sqrt{N} \leqsim M \leq \sqrt{N'} \leqsim \sqrt{N}$. Bounding $N^{-\frac{1}{2}} \leqsim M^{-1}$ and appealing to standard results about Weyl sums (see Lemma 1.1.16 of \cite{tao2012higher}, for example), it follows that there exists an integer $q \leqsim \delta^{-O(1)}$ such that $\|q \xi\|_\T \leqsim \delta^{-O(1)} M^{-2} \sim \delta^{-O(1)} N^{-1}$. However, if $\xi$ does not lie in the major arc centred at 0 then for $N$ sufficiently large we may bound $\|\xi\|_\T \geqsim \delta$, hence $|D_{2m+1}(\xi)| \leqsim \|\xi\|_\T^{-1} \leqsim \delta^{-1}$. Thus in this case by the triangle inequality we have \[ \delta \leqsim \left| \frac{1}{\sqrt{N}} \sum_{m \leq \sqrt{N}} e(\xi m^2+\zeta m) \cdot \frac{D_{2m+1}(\xi)}{\sqrt{N}} \right| \leqsim \frac{\delta^{-1}}{\sqrt{N}}, \] which is a contradiction for $N \geqsim \delta^{-O(1)}$ sufficiently large. It follows that $\xi$ lies in the major arc centred at 0, i.e. $\|\xi\|_\T \leqsim \delta^{-O(1)} N^{-1}$.

\end{proof}

We now have all the lemmas needed to prove our degree lowering result.

\begin{theorem}[Degree Lowering] \label{thm:degreelowering}
Let $N \geq 1$ and $0<\delta \leq 1$. Let $f_0,f_1\:\Z \to \C$ be $1$-bounded functions supported on $[[N_0]]$ for some $N_0 \sim N$ and suppose $\|\Tilde{A}_N^{**}(f_0,f_1) \|_{U^s(\Z)}^{2^s} \geq \delta N^{s+1}$ for some $s \geq 3$. Then either $N \leqsim \delta^{-O(1)}$ or $\|\Tilde{A}_N^{**}(f_0,f_1) \|_{U^{s-1}(\Z)}^{2^{s-1}} \geqsim_s \delta^{O_s(1)} N^s$.
\end{theorem}

\begin{proof}
By assumption we have $\delta N^{s+1} \leq \|\Tilde{A}_N^{**}(f_0,f_1) \|_{U^s(\Z)}^{2^s}$ so using (\ref{eqn:Gowersintermsofgowers}) we have \[ \delta N^{s+1} \leq \sum_{h_1, \dotsc, h_{s-2} \in \Z} \|\Delta_{h_1, \dotsc, h_{s-2}} \Tilde{A}_N^{**}(f_0,f_1) \|_{U^2(\Z)}^4.\] By the popularity principle there exists a set $X \subseteq [[N_0]]^{s-2}$ with $|X| \geqsim \delta N^{s-2}$ such that for all $h \in X$ we have $\delta N^3 \leq \|\Delta_h \Tilde{A}_N^{**}(f_0,f_1) \|_{U^2(\Z)}^{4}$. Applying Lemma \ref{lem:U2inverse}, it follows that there is a function $\xi\:X \to \C$ such that for all $h \in X$ we have \begin{equation} \label{eqn:degreeloweringU2} \delta^{O(1)} N \leqsim \left| \sum_{x \in \Z} \Delta_h \Tilde{A}_N^{**}(f_0,f_1)(x) e(\xi(h) x) \right|. \end{equation} Let $Q = \lceil C_s \delta^{-O(1)} N \rceil$ where $C_s>0$ is an absolute constant to be determined later and for each $h \in X$ denote the element of $\Z/Q\Z$ closest to $\xi(h)$ by $\xi_Q(h)$. Then by the Lipschitz nature of the exponential we have \[ \sum_{x \in \Z} \Delta_h \Tilde{A}_N^{**}(f_0,f_1)(x) e(\xi(h)x) = \sum_{x \in \Z} \Delta_h \Tilde{A}_N^{**}(f_0,f_1)(x) e(\xi_Q(h)x) + O(C_s^{-1} \delta^{O(1)} N). \] Thus for $C_s>0$ taken sufficiently large, we can assume that $\xi$ takes values in $\Z/Q\Z$ without loss of generality. Summing over all $h \in X$ and applying Lemma \ref{lem:dualdifferenceinterchange} we obtain \[ \delta^{O_s(1)} N^{2s-3} \leqsim_s \sum_{h^{(0)}, h^{(1)} \in X} \left| \sum_{x \in \Z} \E_{n \in [N]} \Delta_{h^{(0)}-h^{(1)}}[ 1_{n>\frac{N}{2}} f_0(x+n) f_1(x-\lfloor \sqrt{n} \rfloor+n) ] e(\xi(h^{(0)};h^{(1)})x) \right|. \] We apply the pigeonhole principle to extract $h^{(1)} \in X$ such that \[ \delta^{O_s(1)} N^{s-1} \leqsim_s \sum_{h^{(0)} \in X} \left| \sum_{x \in \Z} \E_{n \in [N]} \Delta_{h^{(0)}-h^{(1)}}[ 1_{n>\frac{N}{2}} f_0(x+n) f_1(x-\lfloor \sqrt{n} \rfloor+n) ] e(\xi(h^{(0)};h^{(1)})x) \right| \] and then use the popularity principle to deduce the existence of a set $Y \subseteq X$ with \[ |Y| \geqsim_s \delta^{O_s(1)} |X| \geqsim \delta^{O_s(1)} N^{s-2} \] such that for all $h^{(0)} \in Y$ we have \begin{equation} \label{eqn:degreeloweringpigeonhole} \delta^{O_s(1)} N \leqsim_s \left| \sum_{x \in \Z} \E_{n \in [N]} \Delta_{h^{(0)}-h^{(1)}}[ 1_{n>\frac{N}{2}} f_0(x+n) f_1(x-\lfloor \sqrt{n} \rfloor+n) ] e(\xi(h^{(0)};h^{(1)})x) \right|. \end{equation} We make the change of variables $x \mapsto x-n$ and then apply the Cauchy--Schwarz inequality to deduce \[ \delta^{O_s(1)} N \leqsim_s \sum_{x \in \Z} |\E_{n \in [N]} \Delta_{h^{(0)}-h^{(1)}} [ 1_{n>\frac{N}{2}} f_1(x-\lfloor \sqrt{n} \rfloor) ] e(-\xi(h^{(0)};h^{(1)})n) |^2.  \] By Plancherel's theorem we therefore have \[ \delta^{O_s(1)} N \leqsim_s \int_\T | (\Delta_{h^{(0)}-h^{(1)}} f_1)^{\widehat{}}(\zeta) \E_{n \in [N]} 1_{n>\frac{N}{2}} e(-\zeta \lfloor \sqrt{n} \rfloor - \xi(h^{(0)};h^{(1)})n)|^2 \,d\zeta. \] We now apply Hölder's inequality and Plancherel's theorem again so that \[ \delta^{O_s(1)} N \leqsim_s \| \Delta_{h^{(0)}-h^{(1)}} f_1 \|_{\ell^2(\Z)}^2 \sup_{\zeta \in \T} |\E_{n \in [N]} 1_{n>\frac{N}{2}} e(-\zeta \lfloor \sqrt{n} \rfloor - \xi(h^{(0)};h^{(1)})n)|^2, \] hence there exists $\zeta \in \T$ such that \[ \delta^{O_s(1)} \leqsim_s |\E_{n \in [N]} 1_{n>\frac{N}{2}} e(-\zeta \lfloor \sqrt{n} \rfloor - \xi(h^{(0)};h^{(1)})n)|. \] By Lemma \ref{lem:exponentialsum} it follows that $\| \xi(h^{(0)};h^{(1)})n) \|_\T \leqsim_s \delta^{-O_s(1)} N^{-1}$ as long as $N \geqsim \delta^{-O(1)}$. Recall that $\xi$ takes values in $\Z/Q\Z$ hence so does $\xi(h^{(0)};h^{(1)})$ for all $h^{(0)} \in Y$. Moreover, the number of elements of $\Z/Q\Z$ lying in this major arc centred at the origin is $O_s(Q \delta^{-O_s(1)} N^{-1}) = O_s(\delta^{-O_s(1)})$. We can therefore write $Y$ as a disjoint union of $O_s(\delta^{-O_s(1)})$ sets on which $h^{(0)} \mapsto \xi(h^{(0)};h^{(1)})$ is constant. Hence by the pigeonhole principle there exists some $A \in [Q]$ and a set $H_0 \subseteq Y$ with $|H_0| \geqsim_s \delta
^{O_s(1)} N^{s-2}$ such that for all $h^{(0)} \in H_0$ we have $\xi(h^{(0)};h^{(1)}) = \frac{A}{Q}$. In particular, for all $h^{(0)} \in H$ we have \[ \xi(h^{(0)}) = \frac{A}{Q} - \sum_{\omega \in \{0,1\}^s \setminus \{0\}} (-1)^{|\omega|} \xi(h^{(\omega)}). \] Since $h^{(1)}$ is fixed, we see that $\xi$ is low rank as a function of $h^{(0)}$ on $H_0$. We sum over $H_0$ in (\ref{eqn:degreeloweringU2}) to obtain \[ \delta^{O_s(1)} N^{s-1} \leqsim_s \sum_{h \in \Z^{s-2}} \left| \sum_{x \in \Z} \Delta_h \Tilde{A}_N^{**}(f_0,f_1)(x) e(\psi(h)x) \right| \] for some low rank function $\psi\:\Z^{s-2} \to \C$. Applying Lemma \ref{lem:lowrank} with $s-2$ instead of $s$ and taking $m=s-3$, it follows that $\delta^{O_s(1)} N^s \leqsim_s \|\Tilde{A}_N^{**}(f_0,f_1) \|_{U^{s-1}(\Z)}^{2^{s-1}}$.
\end{proof}

\subsection{Proof of Theorem \ref{thm:inversethm}}

We now combine the results of the last two subsections to prove Theorem \ref{thm:inversethm}.

\begin{proof}[Proof of Theorem \ref{thm:inversethm}]
We first apply Proposition \ref{prop:dualinverse} to deduce that $\|\Tilde{A}_N^{**}(f_0,f_1)\|_{U^5(\Z)}^{2^5} \geqsim \delta^{O(1)}N^6$. Three applications of Theorem \ref{thm:degreelowering} then yield $\|\Tilde{A}_N^{**}(f_0,f_1)\|_{U^2(\Z)}^4 \geqsim \delta^{O(1)}N^3$. By Lemma \ref{lem:U2inverse} it follows that there exists $\xi \in \T$ such that $|(\Tilde{A}_N^{**}(f_0,f_1))^{\widehat{}}(\xi)| \geqsim \delta^{O(1)}N$. Note that \[ (\Tilde{A}_N^{**}(f_0,f_1))^{\widehat{}}(\xi) = \int_\T \hat{f}_0(\xi-\zeta) \hat{f}_1(\zeta) m_{N;\Z}(\zeta,\xi) \,d\zeta, \] where \[ m_{N;\Z}(\zeta,\xi) = \frac{1}{N} \sum_{n=1}^N 1_{n>\frac{N}{2}} e(- \zeta \lfloor \sqrt{n} \rfloor+\xi n). \] Arguing as in the proof of Theorem \ref{thm:degreelowering}, we apply Cauchy--Schwarz and Plancherel's theorem to obtain $\delta^{O(1)} \leqsim \|m_{N;\Z}(\cdot,\xi)\|_{L^\infty(\T)}$, hence by Lemma \ref{lem:exponentialsum} we have $\|\xi\|_\T \leqsim \delta^{-O(1)}N^{-1}$. With this bound on $\xi$ we now manipulate $(\Tilde{A}_N^{**}(f_0,f_1))^{\widehat{}}(\xi)$. Using the definition of $(\Tilde{A}_N^{**}(f_0,f_1))^{\widehat{}}(\xi)$ and making the change of variables $x \mapsto x-n$ we have \[ (\Tilde{A}_N^{**}(f_0,f_1))^{\widehat{}}(\xi) = \sum_{x \in \Z} \E_{n \in [N]} 1_{n>\frac{N}{2}} f_0(x) f_1(x-\lfloor \sqrt{n} \rfloor) e(-x\xi) e(n\xi). \] As in the proof of Lemma \ref{lem:exponentialsum}, we split $[N]$ into intervals on which $\lfloor \sqrt{n} \rfloor$ is constant, so that \[ (\Tilde{A}_N^{**}(f_0,f_1))^{\widehat{}}(\xi) = \frac{1}{N} \sum_{x \in \Z} \sum_{m \leq \sqrt{N}} \sum_{m^2 \leq n<(m+1)^2} 1_{n>\frac{N}{2}} f_0(x)f_1(x-m) e(-x\xi) e(n\xi) + O \left( \frac{1}{\sqrt{N}} N \right). \] Thus if $N \geqsim \delta^{-O(1)}$ then \[ \delta^{O(1)} N \leqsim \left| \frac{1}{N} \sum_{x \in \Z} \sum_{m \leq \sqrt{N}} \sum_{m^2 \leq n<(m+1)^2} 1_{n>\frac{N}{2}} f_0(x) f_1(x-m) e(-x\xi) e(n\xi) \right|. \] Let $\psi_\xi(m) = \frac{1}{\sqrt{N}} \sum_{n=0}^{2m} 1_{n+m^2>\frac{N}{2}} e(n\xi)$. Then \[ \delta^{O(1)}N \leqsim \left| \sum_{x \in \Z} \E_{m \in [\sqrt{N}]} f_0(x) f_1(x-m) e(-x\xi) e(\xi(m^2+m)) \psi_\xi(m) \right|. \] Let $M = \frac{\delta^A}{A} \sqrt{N}$ for some constant $A>0$ to be determined later. Then we have \[ \delta^{O(1)}N \leqsim \left| \frac{M}{\sqrt{N}} \sum_{k=0}^{\frac{\sqrt{N}}{M}-1} \sum_{x \in \Z} \E_{m \in [M]} f_0(x) f_1(x-m-kM) e(-x\xi) e(\xi(m^2+2mkM+m)) \psi_\xi(m+kM) \right| \] (with some terms in the exponential absorbed by the absolute value) so by the pigeonhole principle there exists $k$ such that \[ \delta^{O(1)}N \leqsim \left| \sum_{x \in \Z} \E_{m \in [M]} f_0(x) f_1(x-m-kM) e(-x\xi) e(\xi(m^2+2mkM+m)) \psi_\xi(m+kM) \right|. \] Since \[ |e(\xi(m^2+2mkM+m))-1| \leqsim \frac{\delta^A}{A}, \] it follows that (for $A$ taken sufficiently large) we have \[ \delta^{O(1)}N \leqsim \left| \sum_{x \in \Z} \E_{m \in [M]} f_0(x) f_1(x-m-kM) e(-x\xi) \psi_\xi(m+kM) \right|. \] By the pigeonhole principle, we may also assume \[ \delta^{O(1)}N \leqsim \left| \sum_{x \in \Z} \E_{m \in [M]} 1_{m \geq \frac{\delta^B}{B} M} f_0(x) f_1(x-m-kM) e(-x\xi) \psi_\xi(m+kM) \right| \] for a sufficiently large absolute constant $B>0$ to be determined later. Applying the triangle inequality  and making the change of variables $x \mapsto x+kM$ yields \[ \delta^{O(1)} N \leqsim \sum_{x \in \Z} |\E_{m \in [M]} 1_{m \geq \frac{\delta^B}{B} M} f_1(x-m) \psi_\xi(m+kM)|. \] Using summation by parts, we have \begin{multline*} |\E_{m \in [M]} 1_{m \geq \frac{\delta^B}{B} M} f_1(x-m) \psi_\xi(m+kM)| \leq \frac{|f_1(x-M) \psi_\xi((k+1)M)|}{M} \\ + \left| \E_{M' \in [M]} (\psi_\xi(M'+kM)-\psi_\xi(M'+1+kM)) \sum_{m \leq M'} 1_{m \geq \frac{\delta^B}{B} M} f_1(x-m) \right|. \end{multline*} Since $\|\psi_\xi\|_{\ell^\infty([\sqrt{N}])} \leq 1$ and $f_1$ is bounded, the first term is $O(\frac{1}{M})$. If $N\geqsim \delta^{-O(1)}$, then \[ \delta^{O(1)} N \leqsim \sum_{x \in \Z} \left| \E_{M' \in [M]} (\psi_\xi(M'+kM)-\psi_\xi(M'+1+kM)) \sum_{m \leq M'} 1_{m \geq \frac{\delta^B}{B} M} f_1(x-m) \right|. \] Choose $N' \leq M$ so that \[ \left| \sum_{m \leq N'} 1_{m \geq \frac{\delta^B}{B}M} f_1(x-m) \right| = \sup_{M' \leq M} \left| \sum_{m \leq M'} 1_{m \geq \frac{\delta^B}{B} M} f_1(x-m) \right|. \] Note that $N' \leq M \leq N^\frac{1}{2}$ and the constraint $m \geq \frac{\delta^B}{B}M$ forces $N' \geq \frac{\delta^B}{B} M = \frac{\delta^{A+B}}z{AB} N^\frac{1}{2}$, so $\delta^{O(1)} N^\frac{1}{2} \leqsim N' \leq N^\frac{1}{2}$. Note also that $\|\psi_\xi(\cdot+1)-\psi_\xi \|_{\ell^1([\sqrt{N}])} \leqsim 1$. Thus by the triangle inequality, we have \[ \delta^{O(1)} N \leqsim \sum_{x \in \Z} \frac{1}{M} \left|\sum_{m \in [N']} 1_{m \geq \frac{\delta^B}{B}M} f_1(x-m) \right|. \] Now \[ \sum_{x \in \Z} \frac{1}{M} \sum_{m \in [N']} 1_{m \geq \frac{\delta^B}{B}M} f_1(x-m) = \sum_{x \in \Z} \frac{1}{M} \left|\sum_{m \in [N']} f_1(x-m) \right| + O\left( \frac{\delta^B}{B} N \right), \] so choosing $A$ sufficiently large and then $B$ sufficiently large depending on $A$, we have \[ \delta^{O(1)} N \leqsim \sum_{x \in \Z} \frac{1}{M} \left|\sum_{m \in [N']} f_1(x-m) \right| \leq \sum_{x \in \Z} |\E_{m \in [N']} f_1(x-m)|. \] This completes the proof.
\end{proof}

\section{Minor Arc Single Scale Estimate} \label{sec:minorarc}

With Theorem \ref{thm:inversethm} in hand we can prove a minor arc estimate which allows us to reduce to the case where $f$ and $g$ are Fourier supported on major arcs. Loosely speaking, this is done by proving that if the Fourier transform of $f$ or $g$ vanishes on the relevant major arc centred at the origin then $\Tilde{A}_N(f,g)$ has a small $\ell^1(\Z)$ norm. More precisely, the aim of this section is to prove the following theorem.

\begin{theorem} \label{thm:sobolevestimate}
Let $N \geq 1$ and $l \in \N$, and let $f,g \in \ell^2(\Z)$. Suppose that \begin{enumerate}[(i)]
    \item $\hat{f}$ vanishes on $\M_{\leq N^{-\frac{1}{2}} 2^l}$; or \item $\hat{g}$ vanishes on $\M_{\leq N^{-1} 2^l}$.
\end{enumerate} Then \begin{equation} \label{eqn:sobolevestimate} \| \Tilde{A}_N(f,g) \|_{\ell^1(\Z)} \leqsim (2^{-cl}+N^{-c}) \|f\|_{\ell^2(\Z)} \|g\|_{\ell^2(\Z)}
\end{equation} for some absolute constant $c>0$.
\end{theorem}

We will first prove Theorem \ref{thm:sobolevestimate} assuming (i) holds and then use this to prove it assuming (ii) holds.

\subsection{Proof of Theorem \ref{thm:sobolevestimate} Assuming (i)}

We introduce the dual function \[ \Tilde{A}_N^*(h,g)(x) = \frac{1}{N} \sum_{n=1}^N 1_{n>\frac{N}{2}} h(x+\lfloor \sqrt{n} \rfloor) g(x+\lfloor \sqrt{n} \rfloor-n) \] which we will consider in this subsection. Note that \[ \sum_{x \in \Z} h(x) \Tilde{A}_N(f,g)(x) = \sum_{x \in \Z} f(x) \Tilde{A}_N^*(h,g)(x). \]

We begin by showing that under the assumptions of Theorem \ref{thm:inversethm} (with $f_0,f_1,f_2$ replaced by $h,f,g$ respectively) $f$ correlates with a function $F$ which is Fourier supported on an appropriate arc centred at $0$.

\begin{proposition} \label{prop:dualstructureI}
Let $N \geq 1$ and $0<\delta \leq 1$. Let $f,g,h\:\Z \to \C$  be $1$-bounded functions supported on $[[N_0]]$ for some $N_0 \sim N$ and suppose \[ \left| \sum_{x \in \Z} h(x) \Tilde{A}_N(f,g)(x) \right| \geq \delta N. \] Then there exists $F \in \ell^2(\Z)$ with $\widehat{F}$ supported on $[[O(\delta^{-O(1)} N^{-\frac{1}{2}})]]$, satisfying $\|F\|_{\ell^\infty(\Z)} \leqsim 1$, $\|F\|_{\ell^1(\Z)} \leqsim N$, and \[ \left|\sum_{x \in \Z} f(x)F(x) \right| \geqsim \delta^{O(1)}N.\]
\end{proposition}

\begin{proof}
We may assume $N \geq A \delta^{-A}$ for a sufficiently large $A>0$ because otherwise one may take $F =\Tilde{A}_N^*(h,g)$. We can therefore apply Theorem \ref{thm:inversethm} to obtain an $N' \in \N$ with $\delta^{O(1)} N^\frac{1}{2} \leqsim N' \leq N^\frac{1}{2}$ and \[ \sum_{|x| \leqsim N} | \E_{n \in [N']} f(x-n)| \geqsim \delta^{O(1)}N. \] Note we are able to restrict the range that we sum $x$ over using the support hypothesis on $f$. By smoothly approximating $1_{[N']}$ we may write \[ \E_{n \in [N']} f(x-n) = \varepsilon \sum_{n \in \Z} \E_{m \in [N']} \F_\R^{-1} \Psi (\varepsilon(n-m)) f(x-n) + O \left(\varepsilon^{10} + \frac{1}{\varepsilon N'} \right) \] with $\varepsilon = B \delta^{-B} N^{-\frac{1}{2}}$ for some $B>0$. Choosing $B>0$ sufficiently large and then $A$ sufficiently large depending on $B$ we obtain \[ \sum_{|x| \leqsim N} \left| \varepsilon \sum_{n \in \Z} \E_{m \in [N']} \F_\R^{-1} \Psi(\varepsilon (n-m)) f(x-n) \right| \geqsim \delta^{O(1)}N. \] We let $\alpha(x)$ denote the conjugate phase of the inner sum and define $\phi(x) = e(\alpha(x)) 1_{|x| \leqsim N}$. Then \[ \left| \sum_{x \in \Z} \phi(x) \varepsilon \sum_{n \in \Z} \E_{m \in [N']} \F_\R^{-1} \Psi(\varepsilon(n-m)) f(x-n) \right| \geqsim \delta^{O(1)}N \] so we may take \[ F(x) = \varepsilon \sum_{n \in \Z} \E_{m \in [N']} \F_\R^{-1} \Psi(\varepsilon(n-m)) \phi(x+n). \] Since \[ \widehat{F}(\xi) = \hat{\phi}(\xi) \E_{m \in [N']} e(m\xi) \Psi \left(\frac{\xi}{\varepsilon} \right) \] is supported on $[-\varepsilon,\varepsilon]$ and $\varepsilon \sim \delta^{-O(1)} N^{-\frac{1}{2}}$, we see that $F$ has the desired properties.
\end{proof}

The next step is to use an application of the Hahn--Banach theorem in order to obtain a decomposition of the dual function $\Tilde{A}_N^*(h,g)$.

\begin{corollary} \label{cor:dualdecomp}
Let $N \geq 1$ and $0<\delta \leq 1$. Let $g,h\:\Z \to \C$ be $1$-bounded functions supported on $[[N_0]]$ with $N_0 \sim N$. Then there exists a decomposition \begin{equation} \Tilde{A}_N^*(h,g) = F+\varepsilon_1+\varepsilon_2
\end{equation} where $F$ satisfies the properties of the function $F$ in Proposition \ref{prop:dualstructureI} and the error terms $\varepsilon_1 \in \ell^1(\Z)$ and $\varepsilon_2 \in \ell^2(\Z)$ satisfy $\|\varepsilon_1\|_{\ell^1(\Z)} \leq \delta N$ and $\|\varepsilon_2\|_{\ell^2(\Z)} \leq \delta$.
\end{corollary}

\begin{proof}
For any $1$-bounded function $f \in \ell^2(\Z)$ supported on $[[2N_0]]$ satisfying \[ \left| \sum_{x \in \Z} f(x) \Tilde{A}_N^*(h,g)(x) \right|>\delta N \] we know there exists $F$ as in the conclusion of Proposition \ref{prop:dualstructureI} such that $|\sum_{x \in \Z} f(x) F(x)| \geqsim \delta^{O(1)} N$. Denote the set of such $F$ by $\mathscr{F}$ and let $V$ denote the $\ell^2(\Z)$ closure of the set \begin{equation} \label{eqn:closureset} \{ \lambda F \mid F \in \mathscr{F}, |\lambda| \leqsim \delta^{-O(1)} \} \cup \{\varepsilon \in \ell^2(\Z) \mid \|\varepsilon\|_{\ell^1(\Z)} \leq \delta N \}. \end{equation} If $\Tilde{A}_N^*(h,g)$ does not lie in $V$ then by the Hahn--Banach separation theorem and the Riesz Representation theorem there exists $f \in \ell^2(\Z)$ such that $|\sum_{x \in \Z} f(x) \Tilde{A}_N^*(h,g)(x)|>\delta N$ and $|\sum_{x \in \Z} f(x) v(x)| \leq \delta N$ for all $v \in V$. Since $\Tilde{A}_N^*(h,g)$ is supported on $[[2N_0]]$ we may assume that $f$ is too. Moreover, by choosing $v \in V$ to be any $\varepsilon \in \ell^2(\Z)$ with $\|\varepsilon\|_{\ell^1(\Z)} \leq \delta N$ we have $|\sum_{x \in \Z} f(x) \varepsilon(x)| \leq \delta N$, or equivalently \[ \|f\|_{\ell^\infty(\Z)} = \sup_{\substack{\varepsilon \in \ell^2(\Z) \\ \|\varepsilon\|_{\ell^1(\Z)} \leq 1}} \left| \sum_{x \in \Z} f(x) \varepsilon(x) \right| \leq 1, \] so $f$ is $1$-bounded. Thus there exists $F \in \mathscr{F}$ such that $|\sum_{x \in \Z} f(x) F(x)| \geqsim \delta^{O(1)} N$. But by choosing $\lambda = A \delta^{-A}$ for a sufficiently large constant $A>0$, we have $\lambda F \in V$ hence \[ \delta^{O(1)}N \leqsim \lambda^{-1} \left|\sum_{x \in \Z} f(x) \lambda F(x) \right| \leq \frac{1}{A} \delta^{A+1} N. \] This is a contradiction for $A$ sufficiently large hence $\Tilde{A}_N^*(h,g)$ does lie in $V$. The desired decomposition comes from writing the dual function as a convex combination of elements of the set (\ref{eqn:closureset}) plus an $\ell^2(\Z)$ error coming from the fact that $V$ is an $\ell^2(\Z)$ closure.
\end{proof}

This decomposition of $\Tilde{A}_N^*(h,g)$ will be used to obtain an estimate for the $\ell^2(\Z)$ norm of $\Tilde{A}_N^*(h,g)$, once restricted to the complement of the principal major arc. Some exponential decay is gained at the cost of a loss of $N^\frac{1}{2}$ (which will be removed later). The assumption that $g$ and $h$ are $1$-bounded is also removed.

\begin{proposition} \label{prop:dualstructureII}
Let $N \geq 1$ and $l \in \N$ satisfy $N \geqsim 2^{2l}$ and let $g,h\:\Z \to \C$ be bounded functions supported on $[[N_0]]$ with $N_0 \sim N$. Then \[ \| \widecheck{\Psi}_{> N^{-\frac{1}{2}} 2^l}*\Tilde{A}_N^*(h,g) \|_{\ell^2(\Z)} \leqsim 2^{-cl} N^\frac{1}{2} \|h\|_{\ell^\infty(\Z)} \|g\|_{\ell^\infty(\Z)} \] for some some absolute constant $c>0$.
\end{proposition}

\begin{proof}
We may normalise so that $g$ and $h$ are $1$-bounded. We then apply Corollary \ref{cor:dualdecomp} with $\delta = 2^{-c_0l}$ for some absolute constant $c_0>0$ to be determined to obtain a decomposition $\Tilde{A}_N^*(h,g) = F+\varepsilon_1+\varepsilon_2$ (keeping notation from that corollary). Note that by choosing $c_0$ appropriately, the support of $\widehat{F}$ is contained in $\M_{\leq N^{-\frac{1}{2}} 2^{l-1}}$, on which $\Psi_{\leq N^{-\frac{1}{2}} 2^l}$ is identically equal to 1. Thus \[ \widecheck{\Psi}_{> N^{-\frac{1}{2}} 2^l}*\Tilde{A}_N^*(h,g) = \widecheck{\Psi}_{> N^{-\frac{1}{2}} 2^l}*\varepsilon_1 + \widecheck{\Psi}_{> N^{-\frac{1}{2}} 2^l}*\varepsilon_2. \] We use the triangle inequality in $\ell^2(\Z)$, (\ref{eqn:psiboundednessZ}), and the fact that $\|\varepsilon_2\|_{\ell^2(\Z)} \leq \delta$ to obtain \[ \|\widecheck{\Psi}_{> N^{-\frac{1}{2}} 2^l}*\varepsilon_2\|_{\ell^2(\Z)} \leqsim \delta. \] We now estimate the remaining term. By the triangle inequality in $\ell^1(\Z)$ we have $\|\varepsilon_1\|_{\ell^\infty(\Z)} \leqsim \delta^{-O(1)}$. Interpolating with the estimate $\|\varepsilon_1\|_{\ell^1(\Z)} \leq \delta N$, we see that \[ \|\varepsilon_1\|_{\ell^q(\Z)} \leqsim_q \delta^{\frac{1}{q} - O(1-\frac{1}{q})} N^\frac{1}{q} \] holds for any $q \in (1,\infty)$. In particular, for an absolute $q \in (1,2)$ sufficiently close to 1, we have $\|\varepsilon_1\|_{\ell^q(\Z)} \leqsim_q \delta^\frac{1}{2} N^\frac{1}{q}$. We therefore see that \[ \|\widecheck{\Psi}_{> N^{-\frac{1}{2}} 2^l}*\varepsilon_1\|_{\ell^q(\Z)} \leqsim_q \delta^\frac{1}{2} N^\frac{1}{q}. \] Since $\Tilde{A}_N^*(h,g)$ is 1-bounded and supported on $[[2N_0]]$, we also have \[ \|\Tilde{A}_N^*(h,g) \|_{\ell^{q'}(\Z)} \leqsim_q N^\frac{1}{q'} \] hence \[ \|\widecheck{\Psi}_{> N^{-\frac{1}{2}} 2^l}*\Tilde{A}_N^*(h,g)\|_{\ell^{q'}(\Z)} \leqsim_q N^\frac{1}{q'}. \] Additionally, since $\|\varepsilon_2\|_{\ell^2(\Z)} \leq \delta$ we have $\|\varepsilon_2\|_{\ell^{q'}(\Z)} \leqsim \delta$. Thus by the triangle inequality in $\ell^{q'}(\Z)$, we obtain \[ \|\widecheck{\Psi}_{> N^{-\frac{1}{2}} 2^l}*\varepsilon_1\|_{\ell^{q'}(\Z)} \leqsim_q N^\frac{1}{q'}. \] Interpolation yields \[ \| \widecheck{\Psi}_{> N^{-\frac{1}{2}} 2^l}*\varepsilon_1 \|_{\ell^2(\Z)} \leqsim_q \delta^\frac{1}{4} N^\frac{1}{2}, \] and hence \[ \|\widecheck{\Psi}_{> N^{-\frac{1}{2}} 2^l}*\Tilde{A}_N^*(h,g)\|_{\ell^2(\Z)} \leqsim_q \delta^\frac{1}{4} N^\frac{1}{2} = 2^{-\frac{c'}{4}l} N^\frac{1}{2}, \] completing the proof.
\end{proof}

Before moving on to the next stage in the proof we prove a so-called ``improving estimate". This kind of result has been proved, for example, by Han--Kovač--Lacey--Madrid--Yang \cite{han2020improving}, who proved that one has inequalities of the form \[ 
\left\| \frac{1}{N} \sum_{n=1}^N f(\cdot + P(n)) \right\|_{\ell^q(\Z)} \leqsim_{p,q,P} N^{-d(\frac{1}{p}-\frac{1}{q})} \|f\|_{\ell^p(\Z)} \] for some (small) range of $p \leq q$, whenever $P$ is a polynomial of degree $d \geq 2$. For our purposes we would like to replace the polynomial sequence $P(n)$ by the Hardy field sequence $\lfloor \sqrt{n} \rfloor-n$. While we could appeal to the results in \cite{han2020improving} by performing a change of variables to obtain a polynomial sequence, we will instead prove our improving estimate directly as we will be able to do so without any restriction on the values $1 \leq p \leq q \leq \infty$ one can take. This will be particularly helpful in Subsection \ref{subsec:breakingduality} when we prove variational estimates with $p<1$ (and make use of Lemma \ref{lem:holdernonbanach} whose proof relies on the following improving estimate), as we will be able to take any $p_1,p_2 \in (1,\infty)$.

\begin{proposition} \label{prop:improving}
Let $1 \leq p \leq q \leq \infty$ and let $f \in \ell^p(\Z)$. Then \[ \left\| \frac{1}{N} \sum_{n=1}^N f(\cdot+\lfloor \sqrt{n} \rfloor-n) \right\|_{\ell^q(\Z)} \leqsim_{p,q} N^{-(\frac{1}{p}-\frac{1}{q})} \|f\|_{\ell^p(\Z)}. \]
\end{proposition}

\begin{proof}
Let $B_Nf(x) = \E_{n \in [N]} f(x+\lfloor \sqrt{n} \rfloor-n)$. Note the trivial bound \begin{equation} \|B_Nf\|_{\ell^p(\Z)} \leq \|f\|_{\ell^p(\Z)} \label{eqn:trivialimproving} \end{equation} for any $p \geq 1$ by the triangle inequality and translation invariance. Then by interpolation it suffices to prove that $\|B_Nf\|_{\ell^\infty(\Z)} \leqsim N^{-1} \|f\|_{\ell^1(\Z)}$. Fix $x \in \Z$ and apply the triangle inequality to obtain \[ |B_Nf(x)| \leq \frac{1}{N} \sum_{n=1}^N |f(x+\lfloor \sqrt{n} \rfloor - n)|. \] We can write the right hand side as \[ \frac{1}{N} \sum_{n=1}^N \sum_{k \in \Z} |f(k)| 1_{x+\lfloor \sqrt{n} \rfloor-n=k}, \] and note that the number of $n \in \N$ such that $\lfloor \sqrt{n} \rfloor -n = k-x$ is $O(1)$. So swapping the order of summation, we can bound $\sum_{n=1}^N 1_{x+\lfloor \sqrt{n} \rfloor-n=k} = O(1)$ to see that \[ |B_Nf(x)| \leqsim \frac{1}{N} \sum_{k \in \Z} |f(k)| = N^{-1} \|f\|_{\ell^1(\Z)}. \] This completes the proof.
\end{proof}

With our improving estimate in hand, we now upgrade $\ell^\infty$ control of $g$ to $\ell^2$ control and remove the loss of $N^\frac{1}{2}$, while keeping the assumption that $g$ is supported on an interval.

\begin{corollary} \label{cor:dualstructureIII}
Under the hypotheses of Proposition \ref{prop:dualstructureII} we have \[ \| \widecheck{\Psi}_{> N^{-\frac{1}{2}} 2^l}*\Tilde{A}_N^*(h,g) \|_{\ell^2(\Z)} \leqsim 2^{-cl} \|h\|_{\ell^\infty(\Z)} \|g\|_{\ell^2(\Z)}. \]
\end{corollary}

\begin{proof}
We have the pointwise bound \[ |\Tilde{A}_N^*(h,g)(x)| \leqsim \|h\|_{\ell^\infty(\Z)} B_N(|g|)(x) \] where \[ B_Nf(x) = \frac{1}{N} \sum_{n=1}^N f(x+\lfloor \sqrt{n} \rfloor-n), \] and Proposition \ref{prop:improving} implies \[ \|B_N(|g|) \|_{\ell^2(\Z)} \leqsim_q N^{-(\frac{1}{q}-\frac{1}{2})} \|g\|_{\ell^q(\Z)} \] holds for any $q<2$. Therefore, after applying the triangle inequality in $\ell^2(\Z)$ and (\ref{eqn:psiboundednessZ}), we have \[ \|\widecheck{\Psi}_{> N^{-\frac{1}{2}} 2^l}*\Tilde{A}_N^*(h,g)\|_{\ell^2(\Z)} \leqsim_q \|h\|_{\ell^\infty(\Z)} \|g\|_{\ell^q(\Z)}. \] Interpolating this estimate with the bound \[ \|\widecheck{\Psi}_{> N^{-\frac{1}{2}} 2^l}*\Tilde{A}_N^*(h,g)\|_{\ell^2(\Z)} \leqsim 2^{-cl} N^\frac{1}{2} \|h\|_{\ell^\infty(\Z)} \|g\|_{\ell^\infty(\Z)} \] obtained from Proposition \ref{prop:dualstructureII} yields the result.
\end{proof}

At this stage, we can remove the support hypothesis on $g$ and $h$.

\begin{corollary} \label{cor:dualstructureIV}
Let $g \in \ell^2(\Z)$ and $h \in \ell^\infty(\Z)$. Then \[ \| \widecheck{\Psi}_{> N^{-\frac{1}{2}} 2^l}*\Tilde{A}_N^*(h,g)\|_{\ell^2(\Z)} \leqsim 2^{-cl} \|h\|_{\ell^\infty(\Z)} \|g\|_{\ell^2(\Z)}. \]
\end{corollary}

\begin{proof}
First note that if $g$ is supported on an interval $I$ of length $N$ then $\Tilde{A}_N^*(h,g)$ is supported on $I+[[N]]$ so we may also restrict $h$ to an $O(N)$-neighbourhood of $I$; at this point Corollary \ref{cor:dualstructureIII} may be employed.

In the case where $g$ is not supported on such an interval, we normalise $h$ so that it is 1-bounded. Then form the partition $\mathcal{I} = \{ (iN,(i+1)N] \}_{i \in \Z}$ of $\Z$ and write $g = \sum_{i \in \mathcal{I}} g1_I$. By the previous case, we have \[ \| \widecheck{\Psi}_{> N^{-\frac{1}{2}} 2^l}*\Tilde{A}_N^*(h,g1_I) \|_{\ell^2(\Z)} \leqsim 2^{-cl} \|g\|_{\ell^2(I)} \] for each $I$. Expanding the square, \[ \| \widecheck{\Psi}_{> N^{-\frac{1}{2}} 2^l}*\Tilde{A}_N^*(h,g)\|_{\ell^2(\Z)}^2 \\= \sum_{I,J \in \mathcal{I}} \langle \widecheck{\Psi}_{> N^{-\frac{1}{2}} 2^l}*\Tilde{A}_N^*(h,g1_I), \widecheck{\Psi}_{> N^{-\frac{1}{2}} 2^l}*\Tilde{A}_N^*(h,g1_J) \rangle. \] By the Cauchy--Schwarz inequality and the previous case, for each $I,J \in \mathcal{I}$ we have \begin{equation} \label{eqn:intervaldistance} \langle \widecheck{\Psi}_{> N^{-\frac{1}{2}} 2^l}*\Tilde{A}_N^*(h,g1_I), \widecheck{\Psi}_{> N^{-\frac{1}{2}} 2^l}*\Tilde{A}_N^*(h,g1_J) \rangle \leqsim 2^{-2cl} \|g\|_{\ell^2(I)} \|g\|_{\ell^2(J)}. \end{equation} Moreover, note that $\Tilde{A}_N^*(h,g1_I)$ is supported on $3I$ and similarly for $\Tilde{A}_N^*(h,g1_J)$. Thus we also have \begin{multline*} \langle \widecheck{\Psi}_{> N^{-\frac{1}{2}} 2^l}*\Tilde{A}_N^*(h,g1_I), \widecheck{\Psi}_{> N^{-\frac{1}{2}} 2^l}*\Tilde{A}_N^*(h,g1_J) \rangle \\ \leqsim \left \langle \frac{d(I,J)}{N} \right \rangle^{-4} \|\Tilde{A}_N^*(h,g1_I)\|_{\ell^2(\Z)} \| \Tilde{A}_N^*(h,g1_J) \|_{\ell^2(\Z)}, \end{multline*} which we can bound above by \[ \left \langle \frac{d(I,J)}{N} \right \rangle^{-4} \|g\|_{\ell^2(I)} \|g\|_{\ell^2(J)} \] using (\ref{eqn:HolderBanach}). Taking geometric means with (\ref{eqn:intervaldistance}), we see that \[ \| \widecheck{\Psi}_{> N^{-\frac{1}{2}} 2^l}* \Tilde{A}_N^*(h,g) \|_{\ell^2(\Z)}^2 \leqsim \sum_{I,J \in \mathcal{I}} 2^{-cl} \left \langle \frac{d(I,J)}{N} \right \rangle^{-2} \|g\|_{\ell^2(I)} \|g\|_{\ell^2(J)}, \] and the result follows after applying the Cauchy--Schwarz inequality in $\ell^2(\mathcal{I}^2)$. 
\end{proof}

Finally, we can prove Theorem \ref{thm:sobolevestimate} assuming (i) holds.

\begin{proposition}
Let $N \geq 1$ and $l \in \N$, and let $f,g \in \ell^2(\Z)$. Suppose that $\hat{f}$ vanishes on $\M_{\leq N^{-\frac{1}{2}} 2^l}$. Then \[ \|\Tilde{A}_N(f,g) \|_{\ell^1(\Z)} \leqsim (2^{-cl}+N^{-c}) \|f\|_{\ell^2(\Z)} \|g\|_{\ell^2(\Z)} \] for some absolute constant $c>0$.
\end{proposition}

\begin{proof}
If $N \leqsim 2^{2l}$, say, then this follows by (\ref{eqn:HolderBanach}). Otherwise, it suffices to show that \[ \|\Tilde{A}_N(f,g)\|_{\ell^1(\Z)} \leqsim 2^{-cl} \|f\|_{\ell^2(\Z)} \|g\|_{\ell^2(\Z)} \] which is equivalent, by duality, to proving that \[ \sum_{x \in \Z} h(x) \Tilde{A}_N(f,g)(x) \leqsim 2^{-cl} \|h\|_{\ell^\infty(\Z)} \|f\|_{\ell^2(\Z)} \|g\|_{\ell^2(\Z)}, \] or \[ \sum_{x \in \Z} f(x) \Tilde{A}_N^*(h,g)(x) \leqsim 2^{-cl} \|h\|_{\ell^\infty(\Z)} \|f\|_{\ell^2(\Z)} \|g\|_{\ell^2(\Z)}, \] for all $h \in \ell^\infty(\Z)$. The support assumption on $\hat{f}$ and Parseval's identity implies \[ \sum_{x \in \Z} f(x) \left[ \widecheck{\Psi}_{\leq N^{-\frac{1}{2}} 2^l}* \Tilde{A}_N^*(h,g) \right](x) =0, \] hence it suffices to prove that \[ \sum_{x \in \Z} f(x) \left[ \widecheck{\Psi}_{> N^{-\frac{1}{2}} 2^l}*\Tilde{A}_N^*(h,g) \right](x) \leqsim 2^{-cl} \|h\|_{\ell^\infty(\Z)} \|f\|_{\ell^2(\Z)} \|g\|_{\ell^2(\Z)}. \] By the Cauchy--Schwarz inequality this is implied by \[ \|\widecheck{\Psi}_{> N^{-\frac{1}{2}}2^l}*\Tilde{A}_N^*(h,g) \|_{\ell^2(\Z)} \leqsim 2^{-cl} \|h\|_{\ell^\infty(\Z)} \|g\|_{\ell^2(\Z)}, \] but this is exactly the content of Corollary \ref{cor:dualstructureIV}.
\end{proof}

\subsection{Proof of Theorem \ref{thm:sobolevestimate} Assuming (ii)}

In this subsection we prove Theorem \ref{thm:sobolevestimate} assuming $\hat{g}$ vanishes on $\M_{\leq N^{-1}2^l}$. The steps are almost identical to those in the previous subsection. We begin by showing that under the hypotheses of Proposition \ref{prop:dualstructureI} the function $g$ also correlates with a function $G$ which is Fourier supported on the major arc centred at $0$. This stage is the only one which is noticeably different from the analogous result in the previous subsection.

\begin{proposition} \label{prop:dualstructureforgI}
Under the hypotheses of Proposition \ref{prop:dualstructureI} there exists $G \in \ell^2(\Z)$ with $\widehat{G}$ supported on $[[O(\delta^{-O(1)} N^{-1})]]$, satisfying $\|G\|_{\ell^\infty(\Z)} \leqsim 1$, $\|G\|_{\ell^1(\Z)} \leqsim N$, and \[ \left|\sum_{x \in \Z} g(x)G(x) \right| \geqsim \delta^{O(1)}N.\]
\end{proposition}

\begin{proof}
Like in the proof of Proposition \ref{prop:dualstructureI} we assume that $N \geq A \delta^{-A}$ for a sufficiently large $A>0$ to be determined later. By assumption we have \[ \delta N \leq \left| \sum_{x \in \Z} h(x) \Tilde{A}_N(f,g)(x) \right| = \left| \sum_{x \in \Z} f(x) \Tilde{A}_N^*(h,g)(x) \right|, \] hence by the Cauchy--Schwarz inequality we have $\|\Tilde{A}_N^*(h,g)\|_{\ell^2(\Z)}^2 \geq \delta^2 N$. We apply Corollary \ref{cor:dualdecomp} with $\delta$ replaced by $c_0\delta$ for a sufficiently small $c_0>0$ to obtain a decomposition $\Tilde{A}_N^*(h,g) = F + \varepsilon_1+\varepsilon_2$. We write \[ \| \Tilde{A}_N^*(h,g) \|_{\ell^2(\Z)}^2 = \langle \Tilde{A}_N^*(h,g),F \rangle + \langle \Tilde{A}_N^*(h,g),\varepsilon_1 \rangle + \langle \Tilde{A}_N^*(h,g),\varepsilon_2, \rangle, \] and note that \[ \langle \Tilde{A}_N^*(h,g),\varepsilon_1 \rangle + \langle \Tilde{A}_N^*(h,g),\varepsilon_2, \rangle \leqsim c_0 \delta^2 N \] by Hölder's inequality. Choosing $c_0$ sufficiently small, we deduce by the pigeonhole principle that \[ |\langle \Tilde{A}_N^*(h,g),F, \rangle| \geqsim \delta^{O(1)} N \] or \[ \left|\sum_{x \in \Z} \E_{n \in [N]} 1_{n>\frac{N}{2}} h(x) F(x-\lfloor \sqrt{n} \rfloor) g(x-n) \right|. \] Choose $M$ so that $F$ is supported on $[-\frac{1}{M},\frac{1}{M}]$ (hence $M \sim_{c_0} \delta^{O(1)} N^\frac{1}{2}$). Then $\Psi(\frac{M}{2} \cdot)$ is equal to 1 on the support of $\widehat{F}$, so \[ F = F* \left(\Psi\left(\frac{M}{2} \cdot \right) \right)^{\widecheck{}} = \frac{2}{M}(F*\F_\R^{-1} \Psi) \left(\frac{2}{M} \cdot \right). \] Therefore, \[ \left|\sum_{x \in \Z} \E_{n \in [N]} 1_{n>\frac{N}{2}} \sum_{m \in \Z} h(x) F(x-m-\lfloor \sqrt{n} \rfloor) g(x-n) \F_\R^{-1} \Psi \left(\frac{2m}{M} \right) \right| \geqsim \delta^{O(1)} N^\frac{3}{2} \] and we may change variables in $m$: \[ \left|\sum_{x \in \Z} \sum_{m \in \Z} h(x) F(x-m) \E_{n \in [N]} 1_{n>\frac{N}{2}} g(x-n) \F_\R^{-1} \Psi \left(\frac{2(m-\lfloor \sqrt{n} \rfloor)}{M} \right) \right| \geqsim \delta^{O(1)} N^\frac{3}{2}. \] Using the rapid decay of $\F_\R^{-1} \Psi$, we can restrict the sum in $m$ to one over $m \leqsim N^\frac{1}{2}$ at the cost of $O((c_0 \delta)^{O(1)} N^\frac{3}{2})$, which is acceptable on choosing $c_0$ sufficiently small. Thus by the pigeonhole principle there exists $m \leqsim N^\frac{1}{2}$ such that \[ \left|\sum_{x \in \Z} h(x) F(x-m) \E_{n \in [N]} 1_{n>\frac{N}{2}} g(x-n) \F_\R^{-1} \Psi \left(\frac{2(m-\lfloor \sqrt{n} \rfloor)}{M} \right) \right| \geqsim \delta^{O(1)} N. \] By the Cauchy--Schwarz inequality it follows that \[ \sum_{x \in \Z} \left| \E_{n \in [N]} 1_{n>\frac{N}{2}} g(x-n) \F_\R^{-1} \Psi \left( \frac{2(m-\lfloor \sqrt{n} \rfloor)}{M} \right) \right|^2 \geqsim \delta^{O(1)} N, \] and applying Plancherel's theorem we see that \[ \int_\T |\hat{g}(\xi) \phi_m(\xi)|^2 \,d\xi \geqsim \delta^{O(1)} N \] where \[ \phi_m(\xi) = \E_{n \in [N]} 1_{n>\frac{N}{2}} \F_\R^{-1}\Psi \left(\frac{2(m-\lfloor \sqrt{n} \rfloor}{M} \right) e(-n\xi). \] Plancherel's theorem also implies that $\|\hat{g}\|_{L^2(\T)}^2 \leq N$, so by the popularity principle, using the measure $|\hat{g}|^2 \,d\xi$ on $\T$, it follows that there is a set $\Xi \subseteq \T$ such that \[ \int_\Xi |\hat{g}(\xi)|^2 \,d\xi \geqsim \delta^{O(1)}N \] and $|\phi_m(\xi)| \geqsim \delta^{O(1)}$ for all $\xi \in \Xi$. Note that if $\delta^{O(1)} \leqsim |\phi_m(\xi)|$, then using summation by parts we have \[ \delta^{O(1)} \leqsim \sup_{N' \leq N} \left| \frac{1}{N} \sum_{n=1}^{N'} 1_{n>\frac{N}{2}} e(-n\xi) \right| = \sup_{N' \leq N} \frac{|D_{N'}(-\xi) - D_\frac{N}{2}(-\xi)|}{N}. \] By the triangle inequality and the bound on the unnormalised Dirichlet kernel, we deduce that $\|\xi\|_\T \leqsim \delta^{-O(1)} N^{-1}$. Thus $\Xi \subseteq [-\frac{1}{M'},\frac{1}{M'}]$ with $M' \sim \delta^{O(1)} N$. Thus \[ \int_{-\frac{1}{M'}}^\frac{1}{M'} |\hat{g}(\xi)|^2 \,d\xi \geqsim \delta^{O(1)}N, \] and Plancherel's theorem yields \[ \sum_{x \in \Z} \left| \frac{1}{M'} \sum_{m \in \Z} g(x-m) \F_\R^{-1} \Psi \left(\frac{2m}{M'} \right) \right|^2 \geqsim \delta^{O(1)}N. \] It follows that $|\sum_{x \in \Z} g(x) G(x)| \geqsim \delta^{O(1)} N$, where \[ G(x) = \frac{1}{(M')^2} \sum_{m,m' \in \Z} g(x-m+m') \F_\R^{-1} \Psi \left(\frac{2m}{M'} \right) \F_\R^{-1} \Psi \left(\frac{2m'}{M'} \right), \] and $G$ has the desired properties.

\end{proof}

The rest of the argument from the previous subsection can then be repeated with $f$ replaced with $g$, the first dual function $\Tilde{A}_N^*(h,g)$ replaced by the second dual function $\Tilde{A}_N^{**}(h,f)$, and most instances of $N^\frac{1}{2}$ replaced by $N$. In particular, we have the following results.

\begin{corollary}
Let $N \geq 1$ and $0<\delta \leq 1$. Let $f,h\:\Z \to \C$ be $1$-bounded functions supported on $[[N_0]]$ with $N_0 \sim N$. Then there exists a decomposition \[ \Tilde{A}_N^{**}(h,f) = G+\varepsilon_1+\varepsilon_2
\] where $G$ satisfies the properties of the function $G$ in Proposition \ref{prop:dualstructureforgI} and the error terms $\varepsilon_1 \in \ell^1(\Z)$ and $\varepsilon_2 \in \ell^2(\Z)$ satisfy $\|\varepsilon_1\|_{\ell^1(\Z)} \leq \delta N$ and $\|\varepsilon_2\|_{\ell^2(\Z)} \leq \delta$.
\end{corollary}

\begin{proposition} \label{prop:dualstructureforgIII}
Let $N \geq 1$ and $l \in \N$ satisfy $N \geqsim 2^{2l}$, and let $f,h\:\Z \to \C$ be bounded functions supported on $[[N_0]]$ with $N_0 \sim N$. Then \[ \| \widecheck{\Psi}_{> N^{-1}2^l}* \Tilde{A}_N^{**}(h,f) \|_{\ell^2(\Z)} \leqsim 2^{-cl} N^\frac{1}{2} \|h\|_{\ell^\infty(\Z)} \|f\|_{\ell^\infty(\Z)} \] for some some absolute constant $c>0$.
\end{proposition}

\begin{corollary}
Under the hypotheses of Proposition \ref{prop:dualstructureforgIII} we have \[ \| \widecheck{\Psi}_{> N^{-1}2^l}* \Tilde{A}_N^{**}(h,f) \|_{\ell^2(\Z)} \leqsim 2^{-cl} \|h\|_{\ell^\infty(\Z)} \|f\|_{\ell^2(\Z)}. \]
\end{corollary}

\begin{corollary}
Let $f \in \ell^2(\Z)$ and $h \in \ell^\infty(\Z)$. Then \[ \| \widecheck{\Psi}_{> N^{-1}2^l}* \Tilde{A}_N^{**}(h,f) \|_{\ell^2(\Z)} \leqsim 2^{-cl} \|h\|_{\ell^\infty(\Z)} \|f\|_{\ell^2(\Z)}. \]
\end{corollary}

\begin{proposition}
Let $N \geq 1$ and $l \in \N$, and let $f,g \in \ell^2(\Z)$. Suppose that $\hat{g}$ vanishes on $\M_{\leq N^{-1} 2^l}$. Then \[ \|\Tilde{A}_N(f,g)\|_{\ell^1(\Z)} \leqsim (2^{-cl}+N^{-c}) \|f\|_{\ell^2(\Z)} \|g\|_{\ell^2(\Z)} \] for some absolute constant $c>0$.
\end{proposition}

With both cases dealt with, the whole proof of Theorem \ref{thm:sobolevestimate} is complete.

\section{Reducing to Major Arcs} \label{sec:majorarc}

In this section we will be able to utilise Theorem \ref{thm:sobolevestimate} to reduce to the case where both $f$ and $g$ are supported on the principal major arc. To facilitate this reduction, we will need three constants $C_0,C_1,C_2 \in \N$ which will be determined by the end of the argument. We allow these constants to depend on $p_1,p_2,p,r,\lambda$. Moreover, we allow $C_1$ to depend on $C_0$ and we allow $C_2$ to depend on $C_0$ and $C_1$. It now suffices to prove that \begin{equation} \label{eqn:boundwithc2} \| V^r(\Tilde{A}_N(f,g))_{N \in \D} \|_{\ell^p(\Z)} \leqsim_{C_2} \|f\|_{\ell^{p_1}(\Z)} \|g\|_{\ell^{p_2}(\Z)}. \end{equation}

For now we will assume no constraints on $p$. Thus we take any $1<p_1,p_2<\infty$ (later we will distinguish between the cases where $p \geq 1$ and $p<1$) and thus we assume $r>c_{p_1,p_2}$. For notational and intuitive ease in this section we will denote $\widecheck{\Psi}_{\leq x}*f$ by $f_{\leq x}$ and similarly for $f_x$ and $f_{>x}$.

We can now assume that $N \geq C_2$ for all $N \in \D$, which is possible due to (\ref{eqn:orderedpartition}), (\ref{eqn:variationvsellr}), and (\ref{eqn:holdernonbanach}). As previously stated, we would like to assume that $f$ and $g$ are supported on principal major arcs. The constant $C_0$ will be chosen so that the width of these arcs are approximately $N^{-\frac{1}{2}} (\log N)^{C_0}$ and $N^{-1} (\log N)^{C_0}$ respectively. More precisely, for each $N$ let $l(N) = \lfloor C_0 \log \log N \rfloor$. Then we would like to approximate $\tilde{A}_N(f,g)$ by $\tilde{A}_N(f_{\leq N^{-\frac{1}{2}} 2^{l(N)}}, g_{\leq N^{-1} 2^{l(N)}})$. By the bilinearity of $\tilde{A}_N$ we have \begin{multline} \label{eqn:majorminor}  \tilde{A}_N(f,g)-\tilde{A}_N(f_{\leq N^{-\frac{1}{2}} 2^{l(N)}}, g_{\leq N^{-1} 2^{l(N)}}) \\= \tilde{A}_N(f_{>N^{-\frac{1}{2}} 2^{l(N)}}, g) + \tilde{A}_N(f_{\leq N^{-\frac{1}{2}} 2^{l(N)}}, g_{>N^{-1} 2^{l(N)}})  \end{multline} Thus by applying (\ref{eqn:variationvsellr}) and the (quasi-)triangle inequality (\ref{eqn:alternativequasi}) in $\ell^p(\Z)$ we see that \begin{multline} \label{eqn:normmajorminor}  \| V^r(\tilde{A}_N(f,g)-\tilde{A}_N(f_{\leq N^{-\frac{1}{2}} 2^{l(N)}}, g_{\leq N^{-1} 2^{l(N)}}))_{N \in \D} \|_{\ell^p(\Z)} \\ \leqsim \sum_{N \in \D} (\| \tilde{A}_N(f_{>N^{-\frac{1}{2}} 2^{l(N)}}, g) \|_{\ell^p(\Z)} + \| \tilde{A}_N(f_{\leq N^{-\frac{1}{2}} 2^{l(N)}}, g_{>N^{-1} 2^{l(N)}}) \|_{\ell^p(\Z)}). \end{multline} We estimate $\| \tilde{A}_N(f_{>N^{-\frac{1}{2}} 2^{l(N)}}, g) \|_{\ell^p(\Z)}$; the second term is estimated similarly. Since $f_{>N^{-\frac{1}{2}} 2^{l(N)}}$ has a Fourier transform which vanishes on $\M_{\leq N^{-\frac{1}{2}} 2^{l(N)-1}}$, we may apply Theorem \ref{thm:sobolevestimate} to bound \[ \| \tilde{A}_N(f_{>N^{-\frac{1}{2}} 2^{l(N)}}, g) \|_{\ell^1(\Z)} \leqsim_{C_0} (\log N)^{-cC_0} \|f\|_{\ell^2(\Z)} \|g\|_{\ell^2(\Z)}. \] Moreover, by (\ref{eqn:holdernonbanach}), the (quasi-)triangle inequality and (\ref{eqn:psiboundednessZ}) we have \[ \|\Tilde{A}_N(f_{>N^{-\frac{1}{2}} 2^{l(N)}}, g) \|_{\ell^q(\Z)} \leqsim \|f\|_{\ell^{q_1}(\Z)} \|g\|_{\ell^{q_2}(\Z)} \] for any $1<q_1,q_2<\infty$ with $\frac{1}{q} = \frac{1}{q_1}+\frac{1}{q_2}$. By interpolation, we therefore have \[ \|\Tilde{A}_N(f_{>N^{-\frac{1}{2}} 2^{l(N)}}, g) \|_{\ell^p(\Z)} \leqsim_{C_0} (\log N)^{-cC_0} \|f\|_{\ell^{p_1}(\Z)} \|g\|_{\ell^{p_2}(\Z)} \] for some (possibly different) $c>0$. Taking $C_0$ sufficiently large yields, say, \[ \|\Tilde{A}_N(f_{>N^{-\frac{1}{2}} 2^{l(N)}}, g) \|_{\ell^p(\Z)} \leqsim_{C_0} (\log N)^{-10} \|f\|_{\ell^{p_1}(\Z)} \|g\|_{\ell^{p_2}(\Z)}. \] By an analogous argument, we also have \[ \|\Tilde{A}_N(f_{\leq N^{-\frac{1}{2}} 2^{l(N)}}, g_{N^{-1} 2^{l(N)}}) \|_{\ell^p(\Z)} \leqsim_{C_0} (\log N)^{-10} \|f\|_{\ell^{p_1}(\Z)} \|g\|_{\ell^{p_2}(\Z)}. \] Substituting this bound into (\ref{eqn:normmajorminor}) and noting that \[ \sum_{N \in \D} (\log N)^{-10} \leqsim 1 \] due to the lacunary nature of $\D$, we obtain \[ \| V^r(\tilde{A}_N(f,g)-\tilde{A}_N(f_{\leq N^{-\frac{1}{2}} 2^{l(N)}}, g_{\leq N^{-1} 2^{l(N)}}))_{N \in \D} \|_{\ell^p(\Z)} \leqsim_{C_0} \|f\|_{\ell^{p_1}(\Z)} \|g\|_{\ell^{p_2}(\Z)}. \] By applying the (quasi-)triangle inequality in $\ell^p(\Z)$ and the triangle inequality in $V^r$, we see that (\ref{eqn:boundwithc2}) is implied by \begin{equation} \label{eqn:majorarcbound} \| V^r(\tilde{A}_N(f_{\leq N^{-\frac{1}{2}} 2^{l(N)}}, g_{\leq N^{-1} 2^{l(N)}}))_{N \in \D} \|_{\ell^p(\Z)} \leqsim_{C_2} \|f\|_{\ell^{p_1}(\Z)} \|g\|_{\ell^{p_2}(\Z)}. \end{equation}

We decompose $f_{\leq N^{-\frac{1}{2}} 2^{l(N)}}$ into a ``low frequency" term supported on $\M_{\leq N^{-\frac{1}{2}} 2^{-C_1}}$ plus a sum of ``high frequency" terms which form a dyadic decomposition (in Fourier space) of the interval \[ \M_{\leq N^{-\frac{1}{2}} 2^{l(N)}} \setminus \M_{\leq N^{-\frac{1}{2}} 2^{-C_1}}. \] In other words, we write \[ f_{\leq N^{-\frac{1}{2}} 2^{l(N)}} = f_{\leq N^{-\frac{1}{2}} 2^{-C_1}} + \sum_{-C_1<l_1 \leq l(N)} f_{N^{-\frac{1}{2}} 2^{l_1}}. \] For ease of notation we define $f_{N;-C_1} = f_{\leq N^{-\frac{1}{2}} 2^{-C_1}}$ and $f_{N;l_1} = f_{N^{-\frac{1}{2}} 2^{l_1}}$ for $l_1>-C_1$. Similarly we decompose \[ g_{\leq N^{-1} 2^{l(N)}} = \sum_{-C_1 \leq l_2 \leq l(N)} g_{N;l_2} \] where $g_{N;-C_1} = g_{\leq N^{-1} 2^{-C_1}}$ and $g_{N;l_2}=g_{N^{-1} 2^{l_2}}$ for $l_2>-C_1$. Set $\D_{l_1,l_2} = \{ N \in \D \mid l_1,l_2 \leq l(N) \}$. If $p \geq 1$ then by the triangle inequality in $\ell^p(\Z)$ and $V^r$ we bound \[ \| V^r(\tilde{A}_N(f_{\leq N^{-\frac{1}{2}} 2^{l(N)}}, g_{\leq N^{-1} 2^{l(N)}}))_{N \in \D} \|_{\ell^p(\Z)} \leqsim_{C_2} \sum_{l_1,l_2 \geq -C_1} \|V^r(\Tilde{A}_N(f_{N;l_1}, g_{N;l_2}))_{N \in \D_{l_1,l_2}} \|_{\ell^p(\Z)}, \] while if $p<1$ then by the quasi-triangle inequality (\ref{eqn:quasitriangleinequality}) in $\ell^p(\Z)$ and the triangle inequality in $V^r$ we have \[ \| V^r(\tilde{A}_N(f_{\leq N^{-\frac{1}{2}} 2^{l(N)}}, g_{\leq N^{-1} 2^{l(N)}}))_{N \in \D} \|_{\ell^p(\Z)} \leqsim_{C_2} \left(\sum_{l_1,l_2 \geq -C_1} \|V^r(\Tilde{A}_N(f_{N;l_1}, g_{N;l_2}))_{N \in \D_{l_1,l_2}} \|_{\ell^p(\Z)}^p \right)^\frac{1}{p}. \] In either case, it now suffices to prove the following theorem in order to establish (\ref{eqn:majorarcbound}).

\begin{theorem} \label{thm:variationalparaproduct}
Let $l_1,l_2 \geq -C_1$. Then for any $1<p_1,p_2<\infty$ with $\frac{1}{p} = \frac{1}{p_1}+\frac{1}{p_2}$ one has \begin{equation} \|V^r(\Tilde{A}_N(f_{N;l_1}, g_{N;l_2}))_{N \in \D_{l_1,l_2}} \|_{\ell^p(\Z)} \leqsim_{C_2} \langle \max(l_1,l_2) \rangle^{O(1)} 2^{-c\max(l_1,l_2) 1_{p_1=p_2=2}} \|f\|_{\ell^{p_1}(\Z)} \|g\|_{\ell^{p_2}(\Z)}. \end{equation}
\end{theorem}

To see why Theorem \ref{thm:variationalparaproduct} suffices, note that by interpolating with the $p_1=p_2=2$ case one obtains \[ \|V^r(\Tilde{A}_N(f_{N;l_1}, g_{N;l_2}))_{N \in \D_{l_1,l_2}} \|_{\ell^p(\Z)} \leqsim_{C_2} \langle \max(l_1,l_2) \rangle^{O(1)} 2^{-c\max(l_1,l_2)} \|f\|_{\ell^{p_1}(\Z)} \|g\|_{\ell^{p_2}(\Z)} \] for a possibly different $c>0$. Therefore, we have \begin{multline*} \| V^r(\tilde{A}_N(f_{\leq N^{-\frac{1}{2}} 2^{l(N)}}, g_{\leq N^{-1} 2^{l(N)}}))_{N \in \D} \|_{\ell^p(\Z)} \\ \leqsim_{C_2} \sum_{l_1,l_2 \geq -C_1} \langle \max(l_1,l_2) \rangle^{O(1)} 2^{-c\max(l_1,l_2)} \|f\|_{\ell^{p_1}(\Z)} \|g\|_{\ell^{p_2}(\Z)} \end{multline*} if $p \geq 1$ and \begin{multline*} \| V^r(\tilde{A}_N(f_{\leq N^{-\frac{1}{2}} 2^{l(N)}}, g_{\leq N^{-1} 2^{l(N)}}))_{N \in \D} \|_{\ell^p(\Z)} \\ \leqsim_{C_2} \left(\sum_{l_1,l_2 \geq -C_1} \langle \max(l_1,l_2) \rangle^{O_p(1)} 2^{-cp\max(l_1,l_2)} \right)^\frac{1}{p} \|f\|_{\ell^{p_1}(\Z)} \|g\|_{\ell^{p_2}(\Z)} \end{multline*} if $p<1$. But the series \[ \sum_{l_1,l_2 \geq -C_1} \langle \max(l_1,l_2) \rangle^{O(1)} 2^{-c\max(l_1,l_2)} \] is $O_{C_2}(1)$ so we are done.

Natural cases arise in Theorem \ref{thm:variationalparaproduct}. Firstly we distinguish between the cases $(p_1,p_2)=(2,2)$ and $(p_1,p_2) \neq (2,2)$, as we require some exponential decay in the form of the constant $2^{-c \max(l_1,l_2)}$ in the former case. We also have different cases depending on whether $l_1=-C_1$ or $l_2=-C_1$. We can deal with one case of Theorem \ref{thm:variationalparaproduct} by using Theorem \ref{thm:sobolevestimate}.

\begin{proposition} \label{prop:hilbertparaproduct}
Let $l_1,l_2>-C_1$. Then \[ \| V^r(\Tilde{A}_N(f_{N;l_1}, g_{N;l_2}))_{N \in \D_{l_1,l_2}} \|_{\ell^1(\Z)} \leqsim_{C_2} 2^{-c \max(l_1,l_2)} \|f\|_{\ell^2(\Z)} \|g\|_{\ell^2(\Z)}. \]
\end{proposition}

\begin{proof}
First use (\ref{eqn:variationvsellr}) to bound \[ \| V^r(\Tilde{A}_N(f_{N;l_1}, g_{N;l_2}))_{N \in \D_{l_1,l_2}} \|_{\ell^1(\Z)} \leqsim \sum_{N \in \D_{l_1,l_2}} \| \Tilde{A}_N(f_{N;l_1}, g_{N;l_2}) \|_{\ell^1(\Z)}. \] Since $l_1,l_2>-C_1$ we have $f_{N;l_1}=f_{N^{-\frac{1}{2}} 2^{l_1}}$ and $g_{N;l_2} = g_{N^{-1} 2^{l_2}}$. In particular, $\hat{f}_{N;l_1}$ and $\hat{g}_{N;l_2}$ vanish on $\M_{\leq N^{-\frac{1}{2}} 2^{l_1-1}}$ and $\M_{\leq N^{-1} 2^{l_2-1}}$ respectively. Thus Theorem \ref{thm:sobolevestimate} yields \[ \|\Tilde{A}_N(f_{N;l_1}, g_{N;l_2})\|_{\ell^1(\Z)} \leqsim (2^{-cl_1}+N^{-c}) \|f_{N;l_1} \|_{\ell^2(\Z)} \|g_{N;l_2} \|_{\ell^2(\Z)} \] and \[ \|\Tilde{A}_N(f_{N;l_1}, g_{N;l_2})\|_{\ell^1(\Z)} \leqsim (2^{-cl_2}+N^{-c}) \|f_{N;l_1} \|_{\ell^2(\Z)} \|g_{N;l_2} \|_{\ell^2(\Z)}. \] Taking geometric means and bounding $N^{-c} \leqsim 2^{-c\max(l_1,l_2)}$ for all $N \in \D_{l_1,l_2}$, we have \[ \|\Tilde{A}_N(f_{N;l_1}, g_{N;l_2})\|_{\ell^1(\Z)} \leqsim 2^{-c\max(l_1,l_2)} \|f_{N;l_1} \|_{\ell^2(\Z)} \|g_{N;l_2} \|_{\ell^2(\Z)}. \] By the lacunary nature of $\D_{l_1,l_2}$ and the fact that $\hat{f}_{N;l_1}$ and $\hat{g}_{N;l_2}$ vanish on principal major arcs, we obtain the pointwise bounds \[ \sum_{N \in \D_{l_1,l_2}} |\hat{f}_{N;l_1}|^2 \leqsim |\hat{f}|^2 \] and \[ \sum_{N \in \D_{l_1,l_2}} |\hat{g}_{N;l_1}|^2 \leqsim |\hat{g}|^2, \] so applying the Cauchy--Schwarz inequality in $\ell^2(\D_{l_1,l_2})$ and Plancherel's theorem completes the proof.
\end{proof}

At this stage we note that \[ \Tilde{A}_N(f_{N;l_1}, g_{N;l_2})(x) = \sum_{y_1,y_2 \in \Z} K_{m_{N;\Z}}^{l_1,l_2}(y_1,y_2) f(x-y_1) g(x-y_2) \] where \[ K_{m_{N;\Z}}^{l_1,l_2}(y_1,y_2) = \int_{\T^2} m_{N;\Z}(\xi_1,\xi_2) \phi_{N;l_1}(\xi_1) \psi_{N;l_2}(\xi_2) e(y_1 \xi_1+y_2 \xi_2) \, d\xi_1 \,d\xi_2 \] with \[ m_{N;\Z}(\xi_1,\xi_2) = \E_{n \in [N]} 1_{n>\frac{N}{2}} e(-\xi_1 \lfloor \sqrt{n} \rfloor-\xi_2 n), \] \[ \phi_{N;l_1} = \begin{cases} \Psi_{N^{-\frac{1}{2}} 2^{l_1}} & l_1>-C_1, \\ \Psi_{\leq N^{-\frac{1}{2}} 2^{l_1}} & l_1 = -C_1, \end{cases} \] and \[ \psi_{N;l_2} = \begin{cases} \Psi_{N^{-1} 2^{l_2}} & l_2>-C_1, \\ \Psi_{\leq N^{-1} 2^{l_2}} & l_2 = -C_1. \end{cases} \] We would like to approximate $m_{N;\Z}$ by $m_{N;\R}$ where \[ m_{N;\R}(\xi_1,\xi_2) = \frac{1}{N} \int_\frac{N}{2}^N e(-\xi_1 \sqrt{t} -\xi_2 t) \,dt, \] and in turn approximate $\Tilde{A}_N(f_{N;l_1},g_{N;l_2})$ by \[ K_{m_{N;\R}}^{l_1,l_2}(y_1,y_2) = \int_{\T^2} m_{N;\R}(\xi_1,\xi_2) \phi_{N;l_1}(\xi_1) \psi_{N;l_2}(\xi_2) e(y_1\xi_1+y_2\xi_2) \,d\xi_1 \,d\xi_2. \] This is the content of the next proposition.

\begin{proposition} \label{prop:continuousmultiplier}
We have \begin{multline*} \left\| \Tilde{A}_N(f_{N;l_1}, g_{N;l_2}) - \sum_{y_1,y_2 \in \Z} K_{m_{N;\R}}^{l_1,l_2}(y_1,y_2) f(\cdot-y_1) g(\cdot-y_2) \right\|_{\ell^p(\Z)} \\ \leqsim_{C_2} 2^{O(\max(1,l_1,l_2))} N^{-\frac{1}{2}} \|f\|_{\ell^{p_1}(\Z)} \|g\|_{\ell^{p_2}(\Z)}. \end{multline*}
\end{proposition}

\begin{proof}
By the preceding discussion we have \begin{multline*} \Tilde{A}_N(f_{N;l_1}, g_{N;l_2}) - \sum_{y_1,y_2 \in \Z} K_{m_{N;\R}}^{l_1,l_2} f(\cdot-y_1) g(\cdot-y_2) \\ = \sum_{y_1,y_2 \in \Z} (K_{m_{N;\Z}}^{l_1,l_2}(y_1,y_2) - K_{m_{N;\R}}^{l_1,l_2}(y_1,y_2)) f(x-y_1) g(x-y_2). \end{multline*} First suppose $p \geq 1$. Then by Minkowski's inequality in $\ell^p(\Z)$, Hölder's inequality, and translation invariance it suffices to show \[ \sum_{y_1,y_2 \in \Z} |K_{m_{N;\Z}}^{l_1,l_2}(y_1,y_2) - K_{m_{N;\R}}^{l_1,l_2}(y_1,y_2)| \leqsim_{C_2} 2^{O(\max(1,l_1,l_2))} N^{-\frac{1}{2}}. \] Write \begin{multline*} K_{m_{N;\Z}}^{l_1,l_2}(y_1,y_2) - K_{m_{N;\R}}^{l_1,l_2}(y_1,y_2) \\ = \int_{\T^2} (m_{N;\Z}-m_{N;\R})(\xi_1,\xi_2) \phi_{N;l_1}(\xi_1) \psi_{N;l_2}(\xi_2) e(y_1\xi_1+y_2\xi_2) \,d\xi_1 \,d\xi_2. \end{multline*} Using integration by parts we see that the following bounds hold: \begin{multline*} K_{m_{N;\Z}}^{l_1,l_2}(y_1,y_2) - K_{m_{N;\R}}^{l_1,l_2}(y_1,y_2) \\ \leqsim \begin{cases} \frac{1}{y_1^2 y_2^2} \int_{\T^2} |\partial^{(2,2)} [ (m_{N;\Z}-m_{N;\R})(\xi_1,\xi_2) \phi_{N;l_1}(\xi_1) \psi_{N;l_2}(\xi_2) ] | \,d\xi_1 \,d\xi_2 & y_1,y_2 \neq 0, \\ \frac{1}{y_1^2} \int_{\T^2} |\partial^{(2,0)} [ (m_{N;\Z}-m_{N;\R})(\xi_1,\xi_2) \phi_{N;l_1}(\xi_1) \psi_{N;l_2}(\xi_2) ] | \,d\xi_1 \,d\xi_2 & y_1 \neq 0, y_2=0, \\ \frac{1}{y_2^2} \int_{\T^2} |\partial^{(0,2)} [ (m_{N;\Z}-m_{N;\R})(\xi_1,\xi_2) \phi_{N;l_1}(\xi_1) \psi_{N;l_2}(\xi_2) ] | \,d\xi_1 \,d\xi_2 & y_1=0, y_2 \neq 0, \\ \int_{\T^2} | (m_{N;\Z}-m_{N;\R})(\xi_1,\xi_2) \phi_{N;l_1}(\xi_1) \psi_{N;l_2}(\xi_2) | \,d\xi_1 \,d\xi_2 & y_1=y_2=0. \end{cases} \end{multline*} Since the series \[ \varepsilon+ \sum_{y \neq 0} \frac{1}{\varepsilon y^2} \] is $O(1)$ for any $0<\varepsilon \leq 1$, we may multiply by a factor of $\varepsilon_1$ if $y_1 \neq 0$ and $\varepsilon_1^{-1}$ if $y_1=0$ into the above bounds, and similarly we can multiply by a factor of $\varepsilon_2$ or $\varepsilon_2^{-1}$ if $y_2 \neq 0$ or $y_2 =0$ respectively. It therefore suffices to prove that \begin{multline} \label{eqn:degreeoffreedom} \sup_{j_1,j_2 \in \{0,2\}} \int_{\T^2} \varepsilon_1^{j_1-1} \varepsilon_2^{j_2-1} |\partial^{(j_1,j_2)} [ (m_{N;\Z}-m_{N;\R})(\xi_1,\xi_2) \phi_{N;l_1}(\xi_1) \psi_{N;l_2}(\xi_2) ] | \,d\xi_1 \,d\xi_2 \\ \leqsim_{C_2} 2^{O(\max(1,l_1,l_2))} N^{-\frac{1}{2}} \|f\|_{\ell^{p_1}(\Z)} \|g\|_{\ell^{p_2}(\Z)} \end{multline} for some $0<\varepsilon_1,\varepsilon_2 \leq 1$. By the Leibniz rule and triangle inequality, the left hand side of (\ref{eqn:degreeoffreedom}) may be bounded above, up to constants, by \begin{multline} \label{eqn:leibnizrule} \sup_{j_1,j_2 \in \{0,2\}} \sum_{(i_1,i_2) \leq (j_1,j_2)} \int_{\T^2} \varepsilon_1^{j_1-1} \varepsilon_2^{j_2-1} |\partial^{(i_1,i_2)}(m_{N;\Z}-m_{N;\R})(\xi_1,\xi_2)| \\ |\partial^{(j_1-i_1)} \phi_{N;l_1}(\xi_1)| |\partial^{(j_2-i_2)} \psi_{N;l_2}(\xi_2)| \,d\xi_1 \,d\xi_2. \end{multline} Note that $\xi_1$ and $\xi_2$ are supported on $\M_{\leq N^{-\frac{1}{2}} 2^{l_1}}$ and $\M_{\leq N^{-1} 2^{l_2}}$ respectively. We fix such $\xi_1,\xi_2$ and $(i_1,i_2) \leq (j_1,j_2)$. Then \begin{multline*} |\partial^{(i_1,i_2)}(m_{N;\Z}-m_{N;\R})(\xi_1,\xi_2)| \\ = \left| \frac{1}{N} \sum_{n=1}^N 1_{n>\frac{N}{2}} \lfloor \sqrt{n} \rfloor^{i_1} n^{i_2} e(-\xi_1 \lfloor \sqrt{n} \rfloor-\xi_2n) - \frac{1}{N} \int_\frac{N}{2}^N \sqrt{t}^{i_1} t^{i_2} e(-\xi_1 \sqrt{t} - \xi_2 t) \,dt \right|. \end{multline*} Let \[ w(t) = t^{\frac{i_1}{2}+i_2} e(-\xi_1\sqrt{t}-\xi_2t). \] Note that \[ \lfloor \sqrt{n} \rfloor^{i_1} n^{i_2} e(-\xi_1 \lfloor \sqrt{n} \rfloor-\xi_2n) = w(n)+O(2^{\max(l_1,l_2)} N^{\frac{i_1}{2}+i_2-\frac{1}{2}}) \] for all $n$, and that \[ w'(t) \leqsim_{C_1} 2^{\max(l_1,l_2)} N^{\frac{i_1}{2}+i_2-1}. \] Thus by the Mean Value Theorem we have \[ \lfloor \sqrt{n} \rfloor^{i_1} n^{i_2} e(-\xi_1 \lfloor \sqrt{n} \rfloor-\xi_2n) - \int_{n-1}^n w(t) \,dt = O_{C_1}(2^{\max(l_1,l_2)} N^{\frac{i_1}{2}+i_2-\frac{1}{2}}), \] and hence \[ |\partial^{(i_1,i_2)}(m_{N;\Z}-m_{N;\R})(\xi_1,\xi_2)| = O_{C_1}(2^{\max(l_1,l_2)} N^{\frac{i_1}{2}+i_2-\frac{1}{2}}). \] Using the identity \[ \int_\T \partial^{(j)} \Psi_{\leq 2^k}(\xi) \,d\xi = 2^{k(1-j)} \int_\T \partial^{(j)} \Psi(\xi) \,d\xi, \] we see that \[ \int_\T\partial^{(j_1-i_1)} \phi_{N;l_1}(\xi_1) \,d\xi_1 \leqsim N^{\frac{j_1}{2}-\frac{i_1}{2}-\frac{1}{2}} 2^{O(\max(1,l_1))} \] and \[ \int_\T \partial^{(j_2-i_2)} \psi_{N;l_2}(\xi_2) \,d\xi_2 \leqsim N^{j_2-i_2-1} 2^{O(\max(1,l_2))}. \] The factors of $N^\frac{j_1}{2}$ and $N^{j_2}$ lost at this stage indicate we should take $\varepsilon_1=N^{-\frac{1}{2}}$ and $\varepsilon_2 = N^{-1}$. Indeed, we have shown that the integral term in (\ref{eqn:leibnizrule}) is $O_{C_1}(2^{O(\max(1,l_1,l_2))} N^{\frac{j_1}{2}+j_2-2})$ so on multiplying by $\varepsilon_1^{j_1-1} = N^{-\frac{j_1}{2}+\frac{1}{2}}$ and $\varepsilon_2^{j_2-1}=N^{-j_2+1}$ we obtain the desired bound.

If $p<1$ then by the quasi-triangle inequality (\ref{eqn:quasitriangleinequality}) and Hölder's inequality it suffices to show \[ \left(\sum_{y_1,y_2 \in \Z} |K_{m_{N;\Z}}^{l_1,l_2}(y_1,y_2) - K_{m_{N;\R}}^{l_1,l_2}(y_1,y_2)|^p \right)^\frac{1}{p} \leqsim_{C_2} 2^{O(\max(1,l_1,l_2))} N^{-\frac{1}{2}}. \] Since the series \[ \varepsilon^p + \sum_{y \neq 0} \frac{1}{\varepsilon^p y^{2p}} \] is $O_p(1)$ for any $p \in (\frac{1}{2},1)$ and $\varepsilon \in (0,1]$, we reduce back to (\ref{eqn:degreeoffreedom}) and we are done.
\end{proof}

Since $l_1,l_2 \leq l(N)$ for all $N \in \D_{l_1,l_2}$, we see that $N \geq 2^{2^{\max(l_1,l_2)/C_0}}$ and hence \[ 2^{O(\max(l_1,l_2))} \sum_{N \in \D_{l_1,l_2}} N^{-\frac{1}{2}} \leqsim_{C_2} \langle \max(l_1,l_2) \rangle^{O(1)} 2^{-c \max(l_1,l_2) 1_{p_1=p_2=2}}. \] Therefore, by (\ref{eqn:variationvsellr}) and Proposition \ref{prop:continuousmultiplier} we have \begin{multline*} \left\| V^r \left( \Tilde{A}_N(f_{N;l_1}, g_{N;l_2}) - \sum_{y_1,y_2} K_{m_{N;\R}}^{l_1,l_2}(y_1,y_2) f(\cdot-y_1) g(\cdot-y_2) \right)_{N \in \D_{l_1,l_2}} \right\|_{\ell^p(\Z)} \\ \leqsim_{C_2} \langle \max(l_1,l_2) \rangle^{O(1)} 2^{-c\max(l_1,l_2)1_{p_1=p_2=2}} \|f\|_{\ell^{p_1}(\Z)} \|g\|_{\ell^{p_2}(\Z)}. \end{multline*} By the (quasi)-triangle inequality in $\ell^p(\Z)$ and the triangle inequality in $V^r$ it now suffices to prove that \begin{multline*} \left\| V^r \left( \sum_{y_1,y_2} K_{m_{N;\R}}^{l_1,l_2}(y_1,y_2) f(\cdot-y_1) g(\cdot-y_2) \right)_{N \in \D_{l_1,l_2}} \right\|_{\ell^p(\Z)} \\ \leqsim_{C_2} \langle \max(l_1,l_2) \rangle^{O(1)} 2^{-c\max(l_1,l_2)1_{p_1=p_2=2}} \|f\|_{\ell^{p_1}(\Z)} \|g\|_{\ell^{p_2}(\Z)}. \end{multline*}

The benefit of replacing the discrete multiplier $m_{N;\Z}$ with the continuous analogue $m_{N;\R}$ is that we can change variables to obtain \[ m_{N;\R}(\xi_1,\xi_2) = \int_\frac{1}{2}^1 e(-\xi_1 \sqrt{Nt}-\xi_2Nt) \,dt. \] Thus we may rewrite \begin{equation} \label{eqn:rewriteasproduct} \sum_{y_1,y_2} K_{m_{N;\R}}^{l_1,l_2}(y_1,y_2) f(\cdot-y_1) g(\cdot-y_2) = \int_\frac{1}{2}^1 \F_\R^{-1} \phi_{N;l_1}(\cdot-\sqrt{Nt})*f \cdot \F_\R^{-1}\psi_{N;l_2}(\cdot-Nt)*g \,dt. \end{equation} Thus Theorem \ref{thm:variationalparaproduct} is implied by the following.

\begin{theorem} \label{thm:modeloperatorI}
We have \begin{multline} \label{eqn:modeloperatorI} \left\| V^r\left( \int_\frac{1}{2}^1 \F_\R^{-1} \phi_{N;l_1}(\cdot-\sqrt{Nt})*f \cdot \F_\R^{-1}\psi_{N;l_2}(\cdot-Nt)*g \,dt \right)_{N \in \D_{l_1,l_2}} \right\|_{\ell^p(\Z)} \\ \leqsim_{C_2} \langle \max(l_1,l_2) \rangle^{O(1)} 2^{-c \max(l_1,l_2) 1_{p_1=p_2=2}} \|f\|_{\ell^{p_1}(\Z)} \|g\|_{\ell^{p_2}(\Z)} \end{multline} for any $l_1,l_2 \geq -C_1$.
\end{theorem}

It is at this stage that we distinguish between the $p \geq 1$ and $p<1$ cases. The main reason for this is that in the $p \geq 1$ case one has access to Minkowski's inequality in $\ell^p(\Z)$ while this is not true if $p<1$. We will first prove Theorem \ref{thm:modeloperatorI} in the case $p \geq 1$ and then interpolate against this result to prove the $p<1$ case.

\subsection{The \texorpdfstring{$p \geq 1$}{p greater than 1} Case of Theorem \ref{thm:modeloperatorI}}

In this subsection we aim to prove the following proposition.

\begin{proposition} \label{prop:modeloperatorIpgeq1}
Suppose $p \geq 1$ and $r>2$. Then \begin{multline*} \left\| V^r\left( \int_\frac{1}{2}^1 \F_\R^{-1} \phi_{N;l_1}(\cdot-\sqrt{Nt})*f \cdot \F_\R^{-1}\psi_{N;l_2}(\cdot-Nt)*g \,dt \right)_{N \in \D_{l_1,l_2}} \right\|_{\ell^p(\Z)} \\ \leqsim_{C_2} \langle \max(l_1,l_2) \rangle^{O(1)} 2^{-c \max(l_1,l_2) 1_{p_1=p_2=2}} \|f\|_{\ell^{p_1}(\Z)} \|g\|_{\ell^{p_2}(\Z)} \end{multline*} holds for any $l_1,l_2 \geq -C_1$.
\end{proposition}

We first assume $(p_1,p_2) \neq (2,2)$ as in this case we do not require some decay in $l_1$ and $l_2$. By applying Minkowski's inequality (in both $\ell^p(\Z)$ and $V^r$) we may bound \begin{multline*} \left\| V^r\left( \int_\frac{1}{2}^1 \F_\R^{-1} \phi_{N;l_1}(\cdot-\sqrt{Nt})*f \cdot \F_\R^{-1}\psi_{N;l_2}(\cdot-Nt)*g \,dt \right)_{N \in \D_{l_1,l_2}} \right\|_{\ell^p(\Z)} \\ \leq \int_\frac{1}{2}^1 \| V^r(\F_\R^{-1} \phi_{N;l_1}(\cdot-\sqrt{Nt})*f \cdot \F_\R^{-1}\psi_{N;l_2}(\cdot-Nt)*g)_{N \in \D_{l_1,l_2}} \|_{\ell^p(\Z)} \,dt. \end{multline*} Then Proposition \ref{prop:modeloperatorIpgeq1} is implied by the following proposition.

\begin{proposition} \label{prop:nodecay}
For each $t \in [\frac{1}{2},1]$ we have \begin{multline*} \| V^r(\F_\R^{-1} \phi_{N;l_1}(\cdot-\sqrt{Nt})*f \cdot \F_\R^{-1}\psi_{N;l_2}(\cdot-Nt)*g)_{N \in \D_{l_1,l_2}} \|_{\ell^p(\Z)} \\ \leqsim_{C_2} \langle \max(l_1,l_2) \rangle^{O(1)} \|f\|_{\ell^{p_1}(\Z)} \|g\|_{\ell^{p_2}(\Z)}. \end{multline*}
\end{proposition}

\begin{proof}
By (\ref{eqn:variationofproduct}) and Hölder's inequality it suffices to show \begin{multline*} \| V^r(\F_\R^{-1} \phi_{N;l_1}(\cdot-\sqrt{Nt})*f)_{N \in \D_{l_1,l_2}} \|_{\ell^{p_1}(\Z)} \|V^r\F_\R^{-1}\psi_{N;l_2}(\cdot-Nt)*g)_{N \in \D_{l_1,l_2}} \|_{\ell^{p_2}(\Z)} \\ \leqsim_{C_2} \langle \max(l_1,l_2) \rangle^{O(1)} \|f\|_{\ell^{p_1}(\Z)} \|g\|_{\ell^{p_2}(\Z)}, \end{multline*} which is implied by the bounds \begin{equation} \label{eqn:largescaleFI} \| V^r(\F_\R^{-1} \phi_{N;l_1}(\cdot-\sqrt{Nt})*f)_{N \in \D_{l_1,l_2}} \|_{\ell^{p_1}(\Z)} \leqsim_{C_2} \max(1,l_1)^{O(1)} \|f\|_{\ell^{p_1}(\Z)} \end{equation} and \begin{equation} \label{eqn:largescaleGI}  \|V^r(\F_\R^{-1}\psi_{N;l_2}(\cdot-Nt)*g)_{N \in \D_{l_1,l_2}} \|_{\ell^{p_2}(\Z)} \leqsim_{C_2} \max(1,l_2)^{O(1)} \|g\|_{\ell^{p_2}(\Z)}. \end{equation} We prove the bound for $f$ as the proof is similar for $g$. If $l_1>-C_1$ then $\phi_{N;l_1}$ vanishes at the origin so we use (\ref{eqn:variationvsellr}) to replace $V^r$ by an $\ell^2(\D_{l_1,l_2})$ norm, and apply Theorem \ref{thm:shifted} with $A=2^{-l_1}$, $C=O(1)$, $\lambda_N = -2^{l_1} \sqrt{t}$, $K = 2^{O(\max(1,l_1))}$, and $d=\frac{1}{2}$. If $l_1=-C_1$ then $\Psi_{\leq 1}$ does not vanish at the origin so we can not apply Theorem \ref{thm:shifted} directly. However, the function $\Psi_{\leq 1}(\xi) e(-\sqrt{t} \xi) - \Psi_{\leq 1}(\xi)$ does, so we can repeat the above argument to obtain \[ \| V^r(\F_\R^{-1}( \phi_{N;l_1}(\cdot-\sqrt{Nt})-\phi_{N;l_1})*f)_{N \in \D_{l_1,l_2}} \|_{\ell^{p_1}(\Z)} \leqsim_{C_2} \max(1,l_1)^{O(1)} \|f\|_{\ell^{p_1}(\Z)}. \] Thus by the triangle inequality (\ref{eqn:largescaleFI}) follows from the estimate \[ \| V^r(\F_\R^{-1} \phi_{N;l_1}*f)_{N \in \D_{l_1,l_2}} \|_{\ell^{p_1}(\Z)} \leqsim_{C_2} \|f\|_{\ell^{p_1}(\Z)}. \] But this follows from Theorem 1.1 (combined with Lemma 2.1) of \cite{jones2008strong}.
\end{proof}

We can now assume that $p_1=p_2=2$, which by (\ref{prop:hilbertparaproduct}) means we can assume that $l_1=-C_1$ or $l_2=-C_1$. In fact, Proposition \ref{prop:nodecay} made no assumption that $(p_1,p_2) \neq (2,2)$, so if $l_1=l_2=-C_1$ then $\max(l_1,l_2)=-C_1$ and hence $2^{-c\max(l_1,l_2)} \sim_{C_1}=1$ and we are done in this case too. We can therefore assume that $l_2>l_1=-C_1$ as the proof in the case $l_1>l_2=-C_1$ is similar and will be outlined later. The first step towards obtaining the decay in $l_1$ and $l_2$ is to rewrite $K_{m_{N;\R}}^{l_1,l_2}$. Indeed, using integration by parts, writing $e(-\xi_2Nt) = -\frac{1}{2\pi i\xi_2 N} \frac{d}{dt} e(-\xi_2Nt)$, yields \begin{multline*} \int_\frac{1}{2}^1 e(-\xi_1 \sqrt{Nt}-\xi_2Nt) \,dt \\= \left[-\frac{1}{2\pi i\xi_2 N} e(-\xi_1\sqrt{Nt}-\xi_2Nt) \right]_\frac{1}{2}^1 - \int_\frac{1}{2}^1 \frac{\xi_1}{2\xi_2 \sqrt{Nt}} e(-\xi_1\sqrt{Nt}-\xi_2 Nt) \,dt. \end{multline*} Thus \begin{multline*} K_{m_{N;\R}}^{l_1,l_2}(y_1,y_2) = \left[-\frac{2^{-l_2}}{2\pi i} \F_\R^{-1} \phi_{N;l_1}(y_1-\sqrt{Nt}) \F_\R^{-1} \Tilde{\psi}_{N;l_2}(y_2-Nt)\right]_\frac{1}{2}^1 \\ - 2^{l_1-l_2-1} \int_\frac{1}{2}^1 \F_\R^{-1} \Tilde{\phi}_{N;l_1}(y_1-\sqrt{Nt}) \F_\R^{-1} \Tilde{\psi}_{N;l_2}(y_2-Nt) \frac{1}{\sqrt{t}} \,dt \end{multline*} where \[ \Tilde{\phi}_{N;l_1}(\xi) = \frac{\sqrt{N} \xi}{2^{l_1}} \phi_{N;l_1}(\xi) \] and \[ \Tilde{\psi}_{N;l_2}(\xi) = \frac{2^{l_2}}{N \xi} \psi_{N;l_2}(\xi). \] Note that $\psi_{N;l_2}$ is supported on $\M_{N^{-1} 2^{l_2}}$ since $l_2>-C_1$. In particular, $\psi_{N;l_2}$ is supported away from the origin so the factor of $\frac{1}{\xi}$ is harmless. Thus we can write (\ref{eqn:rewriteasproduct}) as \begin{multline*}  \left[-\frac{2^{-l_2}}{2\pi i} \F_\R^{-1} \phi_{N;l_1}(\cdot-\sqrt{Nt})*f \cdot \F_\R^{-1} \Tilde{\psi}_{N;l_2}(\cdot-Nt)*g \right]_\frac{1}{2}^1 \\ - 2^{l_1-l_2-1} \int_\frac{1}{2}^1 \F_\R^{-1} \Tilde{\phi}_{N;l_1}(\cdot-\sqrt{Nt})*f \cdot \F_\R^{-1} \Tilde{\psi}_{N;l_2}(\cdot-Nt)*g  \,\frac{dt}{\sqrt{t}}. \end{multline*} To handle (\ref{eqn:modeloperatorI}), we now apply the triangle inequality in $\ell^p(\Z)$ and $V^r$ for each of the three terms, apply Minkowski's inequality to the integral term, and bound $\frac{1}{\sqrt{t}}$ by $O(1)$. Thus, recalling that $l_1=-C_1$, we have \begin{multline*} \left\| V^r\left( \int_\frac{1}{2}^1 \F_\R^{-1} \phi_{N;l_1}(\cdot-\sqrt{Nt})*f \cdot \F_\R^{-1}\psi_{N;l_2}(\cdot-Nt)*g \,dt \right)_{N \in \D_{l_1,l_2}} \right\|_{\ell^p(\Z)} \\ \leqsim_{C_2} \left[ 2^{-l_2} \| V^r(\F_\R^{-1} \phi_{N;l_1}(\cdot-\sqrt{Nt})*f \cdot \F_\R^{-1} \Tilde{\psi}_{N;l_2}(\cdot-Nt)*g)_{N \in \D_{l_1,l_2}} \|_{\ell^p{(\Z)}} \right]_\frac{1}{2}^1 \\+ 2^{-l_2} \int_\frac{1}{2}^1 \| V^r(\F_\R^{-1} \Tilde{\phi}_{N;l_1}(y_1-\sqrt{Nt})*f \F_\R^{-1} \Tilde{\psi}_{N;l_2}(y_2-Nt)*g)_{N \in \D_{l_1,l_2}} \|_{\ell^p(\Z)} \,dt. \end{multline*} The proof of Proposition \ref{prop:nodecay} can be easily modified by replacing $\phi_{N;l_1}$ and $\psi_{N;l_2}$ by $\Tilde{\phi}_{N;l_1}$ or $\Tilde{\psi}_{N;l_2}$ if necessary to handle all terms on the right hand side. This proves Proposition \ref{prop:modeloperatorIpgeq1} in the case $l_2>l_1=-C_1$.

The case where $l_1>l_2=-C_1$ is handled similarly. For the corresponding integration by parts we write $e(-\xi_1 \sqrt{Nt}) = -\frac{\sqrt{t}}{\pi i \xi_1 \sqrt{N}} \frac{d}{dt} e(-\xi_1 \sqrt{Nt})$ to obtain \begin{multline*} \int_\frac{1}{2}^1 e(-\xi_1 \sqrt{Nt}-\xi_2 Nt) \,dt = \left[-\frac{\sqrt{t}}{\pi i \xi_1 \sqrt{N}} e(-\xi_1 \sqrt{Nt} -\xi_2 Nt) \right]_\frac{1}{2}^1 \\ -\int_\frac{1}{2}^1 \frac{\xi_2 \sqrt{Nt}}{\xi_1} e(-\xi_1 \sqrt{Nt} -\xi_2 Nt) \,dt + \int_\frac{1}{2}^1 \frac{1}{4\pi i\xi_1 \sqrt{Nt}} e(-\xi_1 \sqrt{Nt} -\xi_2 Nt) \,dt. \end{multline*} Then \begin{multline*} K_{m_{N;\R}}^{l_1,l_2}(y_1,y_2) = \left[ -\frac{2^{-l_1+1}\sqrt{t}}{2\pi i} \F_\R^{-1} \Tilde{\phi}_{N;l_1}(y_1-\sqrt{Nt}) \F_\R^{-1} \psi_{N;l_2}(y_2-Nt) \right]_\frac{1}{2}^1 \\ - 2^{l_2-l_1} \int_\frac{1}{2}^1 \F_\R^{-1} \Tilde{\phi}_{N;l_1}(y_1-\sqrt{Nt}) \F_\R^{-1} \Tilde{\psi}_{N;l_2}(y_2-Nt) \,\sqrt{t}dt \\ + \frac{2^{-l_1-1}}{2\pi i} \int_\frac{1}{2}^1 \F_\R^{-1} \Tilde{\phi}_{N;l_1}(y_1-\sqrt{Nt}) \F_\R^{-1}\psi_{N;l_2}(y_2-Nt) \,\frac{dt}{\sqrt{t}} \end{multline*} where in this case \[ \Tilde{\phi}_{N;l_1}(\xi) = \frac{2^{l_1}}{\sqrt{N} \xi} \] and \[ \Tilde{\psi}_{N;l_2}(\xi) = \frac{N\xi}{2^{l_2}} \psi_{N;l_2}(\xi). \] The argument is the analogous from this point, hence the proof of Proposition \ref{prop:modeloperatorIpgeq1} is complete.

\subsection{The \texorpdfstring{$p<1$}{p less than 1} Case of Theorem \ref{thm:modeloperatorI}} \label{subsec:breakingduality}

Throughout this subsection we assume that $p<1$. The point of this subsection is to prove the following.

\begin{proposition} \label{prop:modeloperatorIplessone}
Suppose $p<1$ and $r>c_{p_1,c_2}$. Then \begin{multline*} \left\| V^r\left( \int_\frac{1}{2}^1 \F_\R^{-1} \phi_{N;l_1}(\cdot-\sqrt{Nt})*f \cdot \F_\R^{-1}\psi_{N;l_2}(\cdot-Nt)*g \,dt \right)_{N \in \D_{l_1,l_2}} \right\|_{\ell^p(\Z)} \\ \leqsim_{C_2} \langle \max(l_1,l_2) \rangle^{O(1)} \|f\|_{\ell^{p_1}(\Z)} \|g\|_{\ell^{p_2}(\Z)}. \end{multline*}
\end{proposition}

In the previous subsection we dealt with the integral directly using Minkowski's inequality. However, we don't have access to a tool like Minkowski's inequality in the $p<1$ case. Instead, by mixed norm interpolation (interpolating with the $p_1=p_2=2$ case of Theorem \ref{prop:modeloperatorIpgeq1}) it suffices to prove that \begin{multline*} \left\| V^\infty\left( \int_\frac{1}{2}^1 \F_\R^{-1} \phi_{N;l_1}(\cdot-\sqrt{Nt})*f \cdot \F_\R^{-1}\psi_{N;l_2}(\cdot-Nt)*g \,dt \right)_{N \in \D_{l_1,l_2}} \right\|_{\ell^p(\Z)} \\ \leqsim_{C_2} 2^{\max(l_1,l_2)} \|f\|_{\ell^{p_1}(\Z)} \|g\|_{\ell^{p_2}(\Z)}. \end{multline*} or equivalently, by (\ref{eqn:Vinfinity}), \begin{multline*} \left\| \sup_{N \in \D_{l_1,l_2}} \left| \int_\frac{1}{2}^1 \F_\R^{-1} \phi_{N;l_1}(\cdot-\sqrt{Nt})*f \cdot \F_\R^{-1}\psi_{N;l_2}(\cdot-Nt)*g \,dt \right| \right\|_{\ell^p(\Z)} \\ \leqsim_{C_2} 2^{\max(l_1,l_2)} \|f\|_{\ell^{p_1}(\Z)} \|g\|_{\ell^{p_2}(\Z)}. \end{multline*} In lieu of Minkowski's inequality we bound the integral above by the triangle inequality and then replace the integral by a supremum. Thus it suffices to prove that \begin{multline*} \left\| \sup_{N \in \D_{l_1,l_2}} \sup_{t \in [\frac{1}{2},1]} \left| \F_\R^{-1} \phi_{N;l_1}(\cdot-\sqrt{Nt})*f \cdot \F_\R^{-1}\psi_{N;l_2}(\cdot-Nt)*g \right| \right\|_{\ell^p(\Z)} \\ \leqsim_{C_2} 2^{\max(l_1,l_2)} \|f\|_{\ell^{p_1}(\Z)} \|g\|_{\ell^{p_2}(\Z)}. \end{multline*} By splitting up the suprema and applying H\"older's inequality it suffices to prove that \begin{equation} \label{eqn:plessonefbound} \left\| \sup_{N \in \D_{l_1,l_2}} \sup_{t \in [\frac{1}{2},1]}  \left| \F_\R^{-1} \phi_{N;l_1}(\cdot-\sqrt{Nt})*f \right| \right\|_{\ell^{p_1}(\Z)} \leqsim_{C_2} 2^{\max(l_1,0)} \|f\|_{\ell^{p_1}(\Z)} \end{equation} and \begin{equation} \label{eqn:plessonegbound} \left\| \sup_{N \in \D_{l_1,l_2}} \sup_{t \in [\frac{1}{2},1]}  \left| \F_\R^{-1} \psi_{N;l_2}(\cdot-Nt)*g \right| \right\|_{\ell^{p_2}(\Z)} \leqsim_{C_2} 2^{\max(l_2,0)} \|g\|_{\ell^{p_2}(\Z)}. \end{equation} But for all $t \in [\frac{1}{2},1]$ and $N \in \D_{l_1,l_2}$ we can bound \[ \left| \F_\R^{-1} \phi_{N;l_1}(\cdot-\sqrt{Nt})*f \right|  \leqsim 2^{l_1} M_{\mathrm{HL}}f(x) + 2^{-l_1}N^\frac{1}{2} \sum_{|y|>N^\frac{1}{2}} \frac{|f(x-y)|}{y^2} \] if  $l_1 \geq 0$ or \[ \left| \F_\R^{-1} \phi_{N;l_1}(\cdot-\sqrt{Nt})*f \right| \leqsim M_{\mathrm{HL}}f(x) + 2^{-l_1}N^\frac{1}{2} \sum_{|y|>2^{-l_1}N^\frac{1}{2}} \frac{|f(x-y)|}{y^2} \] if $l_1<0$, using the rapid decay of $\F_\R^{-1} \Psi$. Thus by the Hardy--Littlewood maximal inequality and Young's inequality we obtain (\ref{eqn:plessonefbound}), and (\ref{eqn:plessonegbound}) follows analogously.

The constraint \[ r>\max(p_1',p_2') \] comes as a result of the mixed norm interpolation. Without loss of generality suppose $\frac{1}{p_1}\geq \frac{1}{p_2}$. Then the point $(\frac{1}{p_1}, \frac{1}{p_2},0)$ lies on the straight line joining $(\frac{1}{2},\frac{1}{2},0)$ and $(1,\frac{\frac{1}{p}-1}{\frac{2}{p_1}-1},0)$, so we can interpolate between the case where $(q_1,q_2,s)$ is sufficiently close to $(2,2,2)$ and where it is sufficiently close to $(1,\frac{\frac{2}{p_1}-1}{\frac{1}{p}-1},\infty)$. The lower bound on $r$ for which one has a variational inequality (for fixed values of $p_1$ and $p_2$) can then be made arbitrarily close to \[ \left(1-\max\left(\frac{1}{p_1},\frac{1}{p_2}\right)\right)^{-1} = \max(p_1',p_2'). \]

This completes the proof of Proposition \ref{prop:modeloperatorIplessone}, and hence of Theorem \ref{thm:reducedmainthm}.

\printbibliography
\end{document}